\nonstopmode
\documentclass[12pt]{amsart}

\usepackage{amsmath}
\usepackage{amscd}
\usepackage{amssymb}
\usepackage{mathtools}
\usepackage{graphicx}
\usepackage[all]{xypic}

\numberwithin{equation}{section}

\theoremstyle{plain} 
\newtheorem{theorem}[equation]{Theorem}
\newtheorem{lemma}[equation]{Lemma}

\newtheorem{proposition}[equation]{Proposition}
\newtheorem{corollary}[equation]{Corollary}

\theoremstyle{definition}
\newtheorem{definition}[equation]{Definition}

\newtheorem{example}[equation]{Example}
\newtheorem{remark}[equation]{Remark}
\newtheorem*{remark*}{Remark}

\newcommand{\classificationtheoremtext}{
  Suppose that $n=mp^{k}$, where $m$ is coprime to~$p$, and suppose
  that $H\subseteq\Un$ is a \pdash toral subgroup of $\Un$ that contains~$S^{1}$.
  Let $\Gamma_{k}$ act on $\Cn$ by $m$ copies of the standard representation of $\Gamma_{k}$ on~$\Cpk$.
\begin{enumerate}
\item If $m=1$, then $H$ is problematic if and only if $H$ is conjugate to a subgroup of~$\Gamma_{k}$.
\item If $m>1$, then $H$ is problematic if and only if $m$ is a power of a prime different from~$p$ and $H$ is conjugate to a coisotropic subgroup of~$\Gamma_{k}$.
\end{enumerate}
}

\newtheorem*{classificationtheorem}
   {Theorem~\ref{theorem: new classification theorem}}

\theoremstyle{plain} 

\newtheorem*{elementaryabeliantheorem}
         {Theorem~\ref{thm: H elementary abelian}}
\newcommand{\elementaryabeliantheoremtext}
{
  Let $H$ be a problematic \pdash toral subgroup of $\Un$ that contains~$S^{1}$. Then
  \begin{enumerate}
  \item  \label{item: elem abelian}
  $H$ is a projective elementary abelian \pdash group, and
  \item \label{item: character}
  the character of $H$ is
  $\begin{cases}
  \character{H}(h)=0 & h\not\in S^{1}\\
  \character{H}(h)=nh & h\in S^{1}.
  \end{cases}$
  \end{enumerate}
}

\newtheorem*{subgroupGammadiag}
{Theorem~\ref{theorem: Non_contractible_implies_subgroup_Gamma_diag}}
\newcommand{\subgroupGammadiagtext}
{ Let $H$ be a problematic \pdash toral subgroup of $\Un$ that contains~$S^{1}$, and
  suppose that $n=mp^{k}$, where $m$ and $p$ are coprime.
  Then $H$ is conjugate to a subgroup of
  $\Gamma_{k}\subset \Un$.  }

\newtheorem*{simplyconnectedthm}
         {Theorem~\ref{theorem: simply connected}}

\newcommand{\simplyconnectedtext}{
If $n\geq 3$, then $\Lcal_{n}$ is simply connected.
}

\newtheorem*{corollarycontractible}{Corollary~\ref{corollary: contractible}}
\newcommand{\corollarycontractibletext}{
The space $\Lcal_{n}$ is contractible if and only if
$n$ is not a power of a prime.
}

\newtheorem*{jointheorem}
         {Theorem~\ref{theorem: join theorem}}

\newcommand{\jointheoremtext}{
Suppose that $n=mp^{k}$ with $k\geq 1$.  Then there is an equivalence of $\Uof{m}$-spaces
\[
    \Lcal_{m}\join\weakfixedspecific{\Gamma_{k}}{p^{k}}
             \rightarrow
             \weakfixedspecific{\Gamma_{k}}{mp^{k}}.
\]
}

\newcommand{\defining}[1]{{\emph{#1}}}
\DeclareMathOperator{\Unif}{Unif}
\DeclareMathOperator{\Isotyp}{Isotyp}
\newcommand{\tensorsubcomplex}{Z}
\newcommand{\Uniform}[2]{\Unif\!\weakfixedspecific{#1}{#2}}
\newcommand{\isotypicsubcomplex}{\Isotyp\!\weakfixedspecific{\Gamma_{k}}{mp^k}}

\newcommand{\first}{X}
\newcommand{\second}{Y}

\newcommand{\isogroupof}[1]{I_{#1}}
\newcommand{\UnIsotropyof}[1]{K_{#1}}
\newcommand{\objcomponent}[1]{C_{#1}}
\newcommand{\morphcomponent}[2]{M_{\typeof{#1\rightarrow #2}}}
\newcommand{\Cyl}[2]{\rm{Cyl}\left(#1\rightarrow #2\right)}
\newcommand{\nund}{{\underline{n}}}
\newcommand{\type}{t}
\newcommand{\typeof}[1]{\type(#1)}
\newcommand{\morphtriangle}[2]{T_{\typeof{#1\rightarrow #2}}}

\newcommand{\reals}{{\mathbb{R}}}
\newcommand{\integers}{{\mathbb{Z}}}

\newcommand{\complexes}{{\mathbb{C}}}
\newcommand{\field}{{\mathbb{F}}}
\newcommand{\join}{\ast}

\newcommand{\Z}{\integers}

\newcommand{\Cpk}{\complexes^{p^{k}}}
\newcommand{\Cn}{\complexes^{n}}
\newcommand{\Upk}{U\kern-2.5pt\left(p^{k}\right)}
\newcommand{\Upp}{U\kern-.2pt(p)}
\newcommand{\Uof}[1]{U\kern-2pt\left(#1\right)}
\newcommand{\Un}{\Uof{n}}
 \newcommand{\character}[1]{\chi_{#1}}
 \newcommand{\repn}[1]{\rho_{#1}}

\newcommand{\Hbarmodp}{\Hbar/\Phi\!\left(\Hbar\right)}

\newcommand{\HH}{\mathbb{H}} 
\newcommand{\TT}{\mathbb{T}} 

\newcommand{\weakfixed}[1]{ \left(\Lcal_{n}\right)^{#1} }
\newcommand{\weakfixedspecific}[2]{ \left(\Lcal_{#2}\right)^{#1} }
\newcommand{\strongfixed}[1]{ \left(\Lcal_{n}\right)^{#1}_{\,\st}}
\newcommand{\strongfixedspecific}[2]{ \left(\Lcal_{#2}\right)^{#1}_{\,\st}}

\newcommand{\isofixed}[1]{\left(\Lcal_{n}\right)^{#1}_{\,\operatorname{iso}}}
\newcommand{\isofixedspecific}[2]{\left(\Lcal_{#2}\right)^{#1}_{\,\operatorname{iso}}}
\newcommand{\glom}[2]{ #1_{\,\st(#2)}}
\newcommand{\isorefine}[2]{ #1_{\,\operatorname{iso}(#2)}}
\newcommand{\DiagGamma}[2]{\operatorname{Diag}_{#1}\kern-2.8pt\left(\Gamma_{#2}\right)}
\newcommand{\sk}{{\rm{sk}}}
\newcommand{\skel}[2]{\sk^{#1}\!\left(#2\right)}
\newcommand{\degenerate}{L}

\newcommand{\definedas}[1]{{\emph{#1}}}

\newcommand{\n}{\mathbf{n}} 
\newcommand{\vvec}[1]{v_{\scriptscriptstyle{\!#1}}}
\newcommand{\vprimevec}[1]{v'_{\scriptscriptstyle{\!#1}}}
\newcommand{\wvec}[1]{w_{\scriptscriptstyle{\!#1}}}

\def\doCal#1{%
\ifx#1\doAllCalEnd\def\doAllCal{\relax}\else%
 \expandafter\edef\csname#1cal\endcsname{{\noexpand\mathcal #1}}\fi}
\def\doAllCal#1{\doCal#1\doAllCal}
\doAllCal ABCDEFGHIJHLMNOPQRSTUVWXYZ\doAllCalEnd

\def\doBar#1{%
\ifx#1\doAllBarEnd\def\doAllBar{\relax}\else%
 \expandafter\edef\csname#1bar\endcsname{{\noexpand\overline{#1}}}\fi}
\def\doAllBar#1{\doBar#1\doAllBar}
\doAllBar ABCDEFGHIJHLMNOPQRSTUVWXYZabcdefghijhlmnopqrstuvwxyz\doAllBarEnd
\def\doWiggle#1{%
\ifx#1\doAllWiggleEnd\def\doAllWiggle{\relax}\else%
 \expandafter\edef\csname#1wiggle\endcsname{{\noexpand\tilde{#1}}}\fi}
\def\doAllWiggle#1{\doWiggle#1\doAllWiggle}
\doAllWiggle ABCDEFGHIJHLMNOPQRSTUVWXYZabcdefghijklmnopqrstuvwxyz\doAllWiggleEnd

\newcommand{\whatever}{\text{--}}

\newcommand{\Cen}{C}
\DeclareMathOperator{\class}{cl}

\DeclareMathOperator{\Cylinder}{Cyl}

\DeclareMathOperator{\End}{End}

\newcommand{\Id}{I}
\DeclareMathOperator{\im}{im}

\DeclareMathOperator{\Morph}{Morph}
\DeclareMathOperator{\Nerve}{Nerve}

\DeclareMathOperator{\Obj}{Obj}
\DeclareMathOperator{\rank}{rk}

\DeclareMathOperator{\st}{st}

\DeclareMathOperator{\tr}{tr}

\newcommand{\pdash}{$p$\kern1.3pt-}

\newcommand{\CentIdent}[1]{C_{0}\!\left(#1\right)}
\newcommand{\nontransitive}[1]{{\rm{NonTran}}\left(\Lcal_{n}\right)^{#1}}
\newcommand{\nontransitivespecific}[2]{{\rm{NonTran}}\left(\Lcal_{#2}\right)^{#1}}
\newcommand{\properspecific}[2]{{\rm{Move}}\left(\Lcal_{#2}\right)^{#1}}
\newcommand{\intersect}{\Ical}
\newcommand{\dimv}{r}

\begin{document}

\title[Problematic subgroups]
{Classification of problematic subgroups of $\Un$}

\author{Julia E.\ Bergner} \address{Department of Mathematics,
  University of Virginia, Charlottesville, VA} \email{bergnerj@member.ams.org}

\author{Ruth Joachimi}
\address{Department of Mathematics and Informatics, University of Wuppertal, Germany}
\email{joachimi@math.uni-wuppertal.de}
\author{Kathryn Lesh}
\address{Department of Mathematics, Union College, Schenectady NY}
\email{leshk@union.edu}
\author{Vesna Stojanoska}
\address{Department of Mathematics, University of Illinois at Urbana-Champaign, Urbana IL}
\email{vesna@illinois.edu}
\author{Kirsten Wickelgren}
\address{School of Mathematics, Georgia Institute of Technology, Atlanta GA}
\email{kwickelgren3@math.gatech.edu}

\date{\today}

\maketitle
\markboth{\sc{Bergner, Joachimi, Lesh, Stojanoska, and Wickelgren}}
{\sc{Problematic subgroups of $\Un$}}

\begin{abstract}
Let $\Lcal_{n}$ denote the topological poset of
decompositions of $\Cn$ into mutually orthogonal subspaces.
We classify \pdash toral subgroups of $\Un$ that can have noncontractible
fixed points under the action of $\Un$ on~$\Lcal_{n}$.
\end{abstract}

\section{Introduction}

Throughout the paper, let $p$ denote a fixed prime. Let $\Pcal_{n}$ be the
$\Sigma_{n}$-space given by the nerve of the
poset category of proper nontrivial partitions of the set
$\{1,\dots,n\}$. In~\cite{ADL2}, the authors compute the Bredon homology
and cohomology groups of $\Pcal_{n}$ with certain \pdash local coefficients.
The computation is part of a program to give a
new proof of the Whitehead Conjecture and the collapse of the homotopy
spectral sequence for the Goodwillie tower of the identity, one
that does not rely on the detailed knowledge of homology used in
\cite{Kuhn-Whitehead}, \cite{Kuhn-Priddy},
and~\cite{Behrens-Goodwillie-Tower}.  For appropriate \pdash local
coefficients, the Bredon homology and cohomology of~$\Pcal_{n}$ turn
out to be trivial when $n$ is not a power of the prime~$p$, and
nontrivial in only one dimension when $n=p^{k}$ (\cite{ADL2}
Theorem~1.1 and Corollary~1.2). A key ingredient of the proof is the
identification of the fixed point spaces of \pdash subgroups of
$\Sigma_{n}$ acting on $\Pcal_{n}$.  If a \pdash subgroup
$H\subseteq\Sigma_{n}$ has noncontractible fixed points, then $H$
gives an obstruction to triviality of Bredon homology, so it is
``problematic'' (see Definition~\ref{defn: problematic} below). It
turns out that only elementary abelian \pdash subgroups of
$\Sigma_{n}$ with free action on $\{1,\dots,n\}$ can have
noncontractible fixed points on~$\Pcal_{n}$
 (\cite{ADL2} Proposition~6.2). The proof
of the main result of \cite{ADL2} then proceeds by showing that, given
appropriate conditions on the coefficients, these problematic
subgroups can be ``pruned'' or ``discarded'' in most cases, resulting
in sparse Bredon homology and cohomology.

In this paper, we carry out the fixed point calculation analogous to
that of \cite{ADL2} in the context of unitary groups. The calculation
is part of a program to establish the conjectured $bu$-analogue of the
Whitehead Conjecture and the conjectured collapse of the homotopy spectral
sequence of the Weiss tower for the functor $V\mapsto B\Uof{V}$
(see~\cite{Arone-Lesh-Crelle} Section~12).  Let $\Lcal_{n}$ denote the
(topological) poset category of decompositions of $\Cn$ into
nonzero proper orthogonal subspaces.  The category $\Lcal_{n}$ was
first introduced in~\cite{Arone-Topology}.  It is internal to the
category of topological spaces: both its object set and its morphism
set have a topology. (See~\cite{Banff1} for a detailed discussion and
examples.)
The action of the unitary group
$\Un$ on $\Cn$ induces a natural action of $\Un$ on
$\Lcal_{n}$, and the Bredon homology of the $\Un$-space $\Lcal_{n}$
plays an analogous role in the unitary context to the part played by
the Bredon homology of the $\Sigma_{n}$-space~$\Pcal_{n}$ in the
classical context.
The complex $\Lcal_{n}$ has a similar flavor to the ``stable building"
or ``common basis complex" studied in~\cite{Rognes-Topology, RognesPreprint}.

In passing from finite groups to compact Lie groups, one usually replaces the notion of finite \pdash groups with that of \pdash toral groups (i.e., extensions
of finite \pdash groups by tori). Our goal in this paper is to
identify the \pdash toral subgroups of $\Un$ that have noncontractible fixed
point spaces on $\Lcal_{n}$. These groups will be the obstructions to
triviality of the Bredon homology and cohomology of $\Lcal_{n}$ with
suitable \pdash local coefficients.  Hence we make the following
definition.

\begin{definition} \label{defn: problematic}
  A closed subgroup $H\subseteq\Un$ is called \definedas{problematic} if the
  fixed point space of $H$ acting on the nerve of $\Lcal_{n}$ is not
  contractible.
\end{definition}

In this paper, we give a complete classification
of problematic \pdash toral subgroups of $\Un$. To state the result,
we recall that there is a family of \pdash toral subgroups
$\Gamma_{k}\subseteq\Upk$ with the following key properties:
(i)~$\Gamma_{k}$ acts irreducibly on $\Cpk$,
and (ii)~$\Gamma_{k}$ is an extension of the central $S^{1}\subseteq\Upk$
by an elementary abelian \pdash group. The subgroups $\Gamma_{k}$
have appeared in numerous works,
such as \cite{Griess}, \cite{JMO}, \cite{Oliver-p-stubborn},
\cite{Arone-Topology}, ~\cite{Arone-Lesh-Crelle}, and~\cite{AGMV}.  The
structure of $\Gamma_{k}\subset\Upk$ is described explicitly in
Section~\ref{section: Gamma_k}, and we call the corresponding action of $\Gamma_{k}$
on $\Cpk$ the \defining{standard action} of~$\Gamma_{k}$.

For any $n=mp^k$, we can consider $\Gamma_{k}$ acting on $\complexes^{mp^{k}}$
by the $m$-fold multiple of the standard action. That is, we choose a fixed isomorphism
$\Cn\cong\complexes^{m}\otimes\Cpk$, and we let $\Gamma_{k}$ act trivially on $\complexes^{m}$, and by the standard action of $\Gamma_{k}$ on~$\Cpk$. A subgroup $H\subseteq\Gamma_{k}$ is coisotropic if it contains the subgroup of $\Gamma_{k}$ that centralizes it, $C_{\Gamma_{k}}(H)\subseteq H$. (These are, in fact, the $p$-centric subgroups of~$\Gamma_{k}$.)
When we write~$S^{1}$, we always mean the center of~$\Un$.

\begin{theorem}\label{theorem: new classification theorem}
\classificationtheoremtext
\end{theorem}

Note that there is no loss of generality in assuming that $S^{1}\subseteq H$ in
Theorem~\ref{theorem: new classification theorem}. This is because
$S^{1}$ acts on $\Cn$ via multiplication by scalars, and fixes any object of~$\Lcal_{n}$.
Further, $H$ is a \pdash toral subgroup of~$\Un$ if and only if $HS^{1}$ is \pdash toral, and $\weakfixed{HS^{1}}=\weakfixed{H}$. Hence $H$ is problematic if and only if $HS^{1}$ is problematic.

The remainder of the introduction is devoted to outlining the proof
of Theorem~\ref{theorem: new classification theorem}.
We begin with the special case $H=S^{1}$.  Since $\weakfixed{S^1}=\Lcal_{n}$,
the question of whether $S^{1}$ is problematic is asking
if $\Lcal_{n}$ is contractible. We only need one new ingredient to answer this question.

\begin{simplyconnectedthm}
\simplyconnectedtext
\end{simplyconnectedthm}

Homology considerations then give us the following corollary, which says that $S^{1}$ is problematic if and only if $n$ is a power of a prime.

\begin{corollarycontractible}
\corollarycontractibletext
\end{corollarycontractible}

For \pdash toral subgroups $H\subseteq\Un$ that strictly contain~$S^{1}$,
the proof of Theorem~\ref{theorem: new classification theorem}
has two major components. The first is a general argument to establish
that if $n=mp^{k}$ (with $m$ and $p$ coprime) and $H$ is a problematic
\pdash toral subgroup of~$\Un$, then $H$ is conjugate to a subgroup of
$\Gamma_{k}\subset\Un$. The second consists of checking which of these
remaining possibilities are actually problematic, by analyzing the resulting fixed point spaces.

To outline the reduction to subgroups of~$\Gamma_{k}\subset\Un$, we need a little more terminology.  The center of $\Un$ is
$S^{1}$, acting on $\Cn$ via scalar multiples of the
identity matrix, and the \definedas{projective unitary group} $P\Un$
is the quotient~$\Un/S^{1}$.  We say that a closed subgroup
$H\subseteq \Un$ is a \definedas{projective elementary abelian \pdash
  group} if the image of $H$ in $P\Un$ is an elementary abelian \pdash
group.  Lastly, if $H$ is a closed subgroup of~$\Un$, we write
$\character{H}$ for the character of $H$ acting on $\Cn$
through the standard action of $\Un$ on $\Cn$.

\begin{elementaryabeliantheorem}
  \elementaryabeliantheoremtext
\end{elementaryabeliantheorem}

The first part of Theorem~\ref{thm: H elementary abelian} was the
principal result of~\cite{Banff1}. For completeness, we give a
streamlined proof in the current work, in order to obtain the second
part of the theorem, the subgroup's character.

The character data of Theorem~\ref{thm: H elementary abelian} allows
us to narrow down the problematic \pdash toral subgroups of $\Un$ to a
very small, explicitly described collection. The subgroups
$\Gamma_{k}\subseteq\Upk$ (see Section~\ref{section: Gamma_k})
satisfy the conclusions of Theorem~\ref{thm: H elementary abelian}.
The next step is to show that they generate all the possibilities.

\begin{subgroupGammadiag}
\subgroupGammadiagtext
\end{subgroupGammadiag}

The strategy of the proof of
Theorem~\ref{theorem: Non_contractible_implies_subgroup_Gamma_diag}
is to use the first part of
Theorem~\ref{thm: H elementary abelian} and bilinear forms to classify
$H$ up to abstract group isomorphism. Then character theory, together
with the second part of Theorem~\ref{thm: H elementary abelian},
allows us to pinpoint the actual conjugacy class of~$H$ in~$\Un$.

After this, the remaining question for the classification theorem is
whether all of the groups named in
Theorem~\ref{theorem: Non_contractible_implies_subgroup_Gamma_diag}
are, in fact, problematic. Our strategy is
to compute the fixed point spaces fairly explicitly. Let $X\join Y$ denote the
join of the spaces $X$ and~$Y$. The following formula was suggested to us by Gregory Arone.

\begin{jointheorem}
\jointheoremtext
\end{jointheorem}

The remainder of the proof of
Theorem~\ref{theorem: new classification theorem} is then a homology
calculation, based on the formula of
Theorem~\ref{theorem: join theorem}.

\bigskip
\mbox{}
\\

The organization of the paper is as follows.

In the first part of the paper, Section~\ref{section: exploration}
and Section~\ref{section: simple connectivity}, we describe some elementary
properties of $\Lcal_{n}$ and prove the simple connectivity result for $n>2$.

The second part of the paper is devoted to establishing the list of candidates for problematic subgroups.
In Section~\ref{section: normal subgroup condition} we give a
normal subgroup condition from which one can deduce contractibility of
a fixed point set~$\weakfixed{H}$. We follow up in
Section~\ref{section: find normal subgroup} with the proof of
Theorem~\ref{thm: H elementary abelian}, by
finding an appropriate normal subgroup unless $H$ is a projective
elementary abelian \pdash subgroup of~$\Un$.
Section~\ref{section: Gamma_k} is expository and discusses the salient
properties of the subgroup $\Gamma_k\subseteq\Upk$. The projective
elementary abelian \pdash subgroups of $\Un$ are classified up to
abstract group isomorphism using bilinear forms
in Section~\ref{section: alternating forms}. They are shown to be isomorphic to $\Gamma_{s}\times\Delta_{t}$ where $\Delta_{t}$ denotes $(\integers/p)^{t}$. In Section~\ref{section: problematic subgroups}, we prove Theorem~\ref{theorem: Non_contractible_implies_subgroup_Gamma_diag}, allowing us to view this abstract group isomorphism as an isomorphism of representations.

The third part of the paper is devoted to checking the
candidates identified in
Theorem~\ref{theorem: Non_contractible_implies_subgroup_Gamma_diag}.
Section~\ref{section: joins} calculates the $U(m)$-equivariant homotopy type of the fixed points of the action of $\Gamma_{s}$ on $\Lcal_{mp^s}$ (Theorem~\ref{theorem: join theorem}). This allows us to compute the fixed points of the action of $\Gamma_s \times \Delta_t$ on $\Lcal_{mp^s}$
by first computing the fixed points under $\Gamma_s$, reducing the problem to the computation of $\Delta_t$-fixed points.
The bulk of Section~\ref{section: proof of classification theorem}
consists of analyzing the problem of $\Delta_{t}$-fixed points, after which we assemble the pieces to prove Theorem~\ref{theorem: new classification theorem}.

\bigskip

\noindent{\bf{Definitions, Notation, and Terminology}}\\
\indent
When we speak of a subgroup of a Lie group, we always mean a closed
subgroup. If we speak of $S^{1}\subseteq\Un$ without any further
description, we mean the center of~$\Un$.

We generally do not distinguish in notation between a category and its
nerve, since the context will make clear which we mean.

We are concerned with actions of subgroups $H\subseteq\Un$
on~$\Cn$; we write $\repn{H}$ for the restriction of the
standard representation of $\Un$ to $H$, and $\character{H}$ for the
corresponding character. We apply the standard terms from
representation theory to $H$ if they apply to~$\repn{H}$.  For
example, we say that $H$ is \definedas{irreducible} if $\repn{H}$ is
irreducible, and we say that $H$ is \definedas{isotypic} if $\repn{H}$
is the sum of all isomorphic irreducible representations of $H$.  We
also introduce a new term: if $H$ is not isotypic, we say that $H$ is
\definedas{polytypic}, as a succinct way to say that ``the action of
$H$ on $\Cn$ is not isotypic.''

A \definedas{decomposition} $\lambda$ of $\Cn$ is an
(unordered) decomposition of $\Cn$ into mutually
orthogonal, nonzero subspaces. We say that $\lambda$ is
\definedas{proper} if it consists of subspaces properly
contained in~$\Cn$.
If  $v_{1}$, \dots, $v_{m}$ are the subspaces in~$\lambda$, then
we call $v_{1}$, \dots, $v_{m}$ the \definedas{components} or
\definedas{classes} of~$\lambda$; we write
$\class(\lambda)$ for $\{v_{1},\dots,v_{m}\}$ when we want to emphasize
the set of subspaces in the decomposition~$\lambda$.

The action of $\Un$ on $\Cn$ induces a natural action of
$\Un$ on $\Lcal_{n}$, and we are interested in fixed points of this
action. If $\lambda$ consists of the subspaces $v_{1}$,\dots, $v_{m}$,
we say that $\lambda$ is
\definedas{weakly fixed} by a subgroup $H\subseteq\Un$ if for every
$h\in H$ and every $v_{i}$, there exists a $j$ such that $hv_{i}=v_{j}$.
(We use the word ``weakly fixed'' here to contrast with the notion
of ``strongly fixed,'' below.)
We write $\weakfixed{H}$ for the full
subcategory of weakly $H$-fixed objects of~$\Lcal_{n}$.
Note that
$\Nerve\left(\left(\Lcal_{n}\right)^{H}\right)
       \cong\left(\strut\Nerve\left(\Lcal_{n}\right)\right)^{H}$,
and we abuse notation by writing $\weakfixed{H}$ for either
the subcategory or its nerve, depending on context.

By contrast, we say that $\lambda$ is \definedas{strongly fixed} by
$H\subseteq\Un$ if for all $i$, we have $Hv_{i}=v_{i}$, that is, each
$v_{i}$ is a representation of~$H$.  We write $\strongfixed{H}$ for
the full subcategory of strongly $H$-fixed objects of $\Lcal_{n}$ (and
for its nerve).

\bigskip

\noindent{\bf{Acknowledgements}}\\
\indent The authors are grateful to Bill Dwyer, Jesper Grodal, and Gregory Arone, for
many helpful discussions about this project. We
also thank Dave Benson, Jeremiah Heller, John Rognes, and David Vogan for helpful
comments on earlier drafts of this document.

The authors thank the Banff International Research Station and
the Clay Mathematics Institute for financial support. The first, third, and fourth authors received partial support from NSF grants
DMS-1105766 and DMS-1352298, DMS-0968251, DMS-1307390 and DMS-1606479, respectively.  The second
author was partially supported by DFG grant HO 4729/1-1, and the fifth
author was partially supported by an AIM 5-year fellowship and
NSF grants DMS-1406380 and DMS-1552730.  Some of
this work was done while the first, fourth, and fifth authors were in
residence at MSRI during the Spring 2014 semester, supported by NSF
grant 0932078 000.

\section{The topological category $\Lcal_{n}$}
\label{section: exploration}

The goal of this section is to offer descriptions and examples of the
object and morphism spaces of $\Lcal_{n}$, to build intuition
for the category. We devote some attention to the fundamental
groups of the connected components of the object and morphism spaces,
since the goal of Section~\ref{section: simple connectivity} is to
establish simple connectivity of the nerve of~$\Lcal_{n}$.
We also refer the reader
to~\cite{Banff1}, where some low-dimensional cases
were studied explicitly. For example, it was established that $\Lcal_{2}$ is
homeomorphic to~$\reals P^{2}$.

From time to time, we refer to the category of unordered partitions
of an integer~$n$, by which we mean the quotient of the action of $\Sigma_{n}$
on the set of partitions of the set $\n=\{1,\ldots,n\}$. We sometimes use a notation for unordered partitions where sizes of components in the partition are underlined, and the multiplicity of components of a certain size is indicated as scalar multiplication. For example, the unordered partition $7=1+3+3$ is denoted
${\underline{1}}+2\cdot{\underline{3}}$, because the partition has two components of size $3$ and one component of size~$1$.

We begin with the object space of $\Lcal_{n}$, denoted by
$\Obj\left(\Lcal_{n}\right)$. The elements are (unordered)
decompositions of $\Cn$ into proper, mutually orthogonal
subspaces.  There is a natural action of $\Un$
on~$\Obj\left(\Lcal_{n}\right)$, and we topologize
$\Obj\left(\Lcal_{n}\right)$ as the disjoint union of orbits of the
$\Un$ action. (This description and another equivalent one are given in~\cite{Banff1}.)

\begin{definition}\mbox{}\hfill  \label{defn: type objects}
\begin{enumerate}
\item If $\lambda\in\Obj\left(\Lcal_{n}\right)$, then the \defining{type}
  of~$\lambda$, denoted by $\typeof{\lambda}$, is the unordered partition of the
  integer $n$ given by the dimensions of the components of $\lambda$.

\item For $\lambda\in\Obj\left(\Lcal_{n}\right)$, we write
 $\UnIsotropyof{\lambda}\subseteq\Un$ for the isotropy subgroup of
  $\lambda$ under the action of~$\Un$, i.e. $\UnIsotropyof{\lambda}$ contains those elements of $\Un$ that weakly fix $\lambda$.

\item \label{item: set with same type}
If $\tau$ is an unordered partition of the integer~$n$,
we gather the decompositions with this type by defining
\begin{align*}
\objcomponent{\tau}&=
  \left\{\lambda\mid\typeof{\lambda}=\tau\right\}
  \subseteq\Obj\left(\Lcal_{n}\right).
\end{align*}
\end{enumerate}
\end{definition}

Since $\Obj\left(\Lcal_{n}\right)$ is topologized as the disjoint
union of $\Un$-orbits, the transitive action of $\Un$ on $n$-frames for
$\Cn$ gives us the following result.

\begin{lemma}   \label{lemma: object component homogeneous}
The subspace $\objcomponent{\typeof{\lambda}}$ is a homogeneous space; specifically,
\[\objcomponent{\typeof{\lambda}}\cong\Un/\UnIsotropyof{\lambda}. \]
The connected components of~$\Obj\left(\Lcal_{n}\right)$
are given by the spaces $\objcomponent{\tau}$ as $\tau$
ranges over unordered partitions of the integer~$n$.
\end{lemma}

\begin{example}
\label{example: component of lines}
  Consider 
  decompositions of $\Cn$ into $n$ mutually orthogonal
  lines.
  An element of the isotropy group of such a decomposition can act by
  $U(1)$ on each of the lines, and can also permute the lines, so the
  isotropy group is~$\Uof{1}\wr\Sigma_{n}$. The connected component of
  $\Obj\left(\Lcal_{n}\right)$ consisting of all decompositions into
  lines is
\[
\objcomponent{n\cdot\underline{1}}\cong\Un/\left(\Uof{1}\wr\Sigma_{n}\right).
\]
\end{example}

Before moving on, we record a generalization of the computation of
Example~\ref{example: component of lines}.

\begin{lemma}    \label{lemma: object isotropy}
If $\lambda\in\Obj\left(\Lcal_{n}\right)$ has type
 $\typeof{\lambda}=k_{1}\cdot\nund_{1}+\dots+k_{j}\cdot\nund_{j}$, then
\[
\UnIsotropyof{\lambda}\cong
   \left(\Uof{n_{1}}\wr\Sigma_{k_{1}}\right)\times\dots\times
    \left(\Uof{n_{j}}\wr\Sigma_{k_{j}}\right).
\]
\end{lemma}

\medskip

Next we consider $\Morph\left(\Lcal_{n}\right)$, the morphism space
of~$\Lcal_{n}$. Between any two objects of $\Obj\left(\Lcal_{n}\right)$,
there is at most one morphism. Hence the source and target maps give
a $\Un$-equivariant monomorphism
\[
\Morph\left(\Lcal_{n}\right)\longrightarrow
    \Obj\left(\Lcal_{n}\right)\times    \Obj\left(\Lcal_{n}\right),
\]
and we give $\Morph\left(\Lcal_{n}\right)$ the subspace topology.

\begin{definition}\mbox{}\hfill  \label{defn: type morphisms}
\begin{enumerate}
\item If $m \colon \lambda\rightarrow\mu$ is a morphism in~$\Lcal_{n}$, the
  \defining{type} of $m$, denoted by $\typeof{m}$ or $\typeof{\lambda\rightarrow\mu}$,
  is the morphism that $m$ induces in the category of unordered partitions of the
  integer~$n$.

\item We define
    $\morphcomponent{\lambda}{\mu}   \subseteq \Morph\left(\Lcal_{n}\right)$ by
\begin{align*}
\morphcomponent{\lambda}{\mu}&=
\left\{\lambda'\rightarrow\mu'\mid
    \typeof{\lambda'\rightarrow\mu'}
       =\typeof{\lambda\rightarrow\mu}\right\}.
\end{align*}
\end{enumerate}
\end{definition}

\begin{example}
Let $\lambda=\{a, b, c, v\}$ be an orthogonal decomposition
of $\complexes^{5}$ into lines $a$, $b$, and $c$ and a two-dimensional
subspace~$v$. There are two different types of morphisms from $\lambda$ to objects in $\objcomponent{{\underline{2}}+{\underline{3}}}$. The first
is:
\begin{align*}
\{a, b, c, v\}&\longrightarrow\{a+ b+ c ,v\}\\
\left({\underline{1}}+{\underline{1}}+{\underline{1}}\right)+{\underline{2}}
&\longrightarrow
{\underline{3}}+{\underline{2}}.
\end{align*}
The second type takes the sum of two of the lines, and adds the third line to the two-dimensional subspace. There are three morphisms of this type that have $\lambda$ as their source:
\begin{align*}
\{a, b, c, v\}&\longrightarrow\{a+ b, c+ v\}\\
\{a, b, c, v\}&\longrightarrow\{a+ c, b+ v\}\\
\{a, b, c, v\}&\longrightarrow\{b+ c, a+ v\}\\
\left({\underline{1}}+{\underline{1}}\right)
      +\left({\underline{1}}+{\underline{2}}\right)
&\longrightarrow
{\underline{2}}+{\underline{3}}.
\end{align*}
However, the three morphisms of the second type are in the same orbit of the action
of $\Un$ on $\Morph\left(\Lcal_{n}\right)$, whereas the first type
of morphism is in a different $\Un$-orbit.
\end{example}

We have a parallel result to that of
Lemma~\ref{lemma: object component homogeneous}.

\begin{lemma}   \label{lemma: type=component}
The subspace $\morphcomponent{\lambda}{\mu}$ is a homogeneous space; specifically,
\[
\morphcomponent{\lambda}{\mu}
   \cong
   \Un/\left(\UnIsotropyof{\lambda}\cap\UnIsotropyof{\mu}\right).
\]
The connected components of~$\Morph\left(\Lcal_{n}\right)$
are given by the spaces $\morphcomponent{\lambda}{\mu}$ as the type of $\lambda\rightarrow\mu$
ranges over the morphisms of unordered partitions of the integer~$n$.
\end{lemma}

\begin{proof}
We must show that $\Un$ acts transitively on $\morphcomponent{\lambda}{\mu}$. Suppose that
$\typeof{\lambda'\rightarrow\mu'}=\typeof{\lambda\rightarrow\mu}$.
Then $\mu$ and $\mu'$ have the same type, so there exists $u\in\Un$
such that $u\,\mu=\mu'$. It is not necessarily the case that $u \lambda = \lambda'$. However, $u\lambda$ does have the same type as $\lambda'$, and both are refinements
of~$\mu'$. Hence there exists $u'\in\UnIsotropyof{\mu'}$ such that $u' u\lambda=\lambda'$.

To compute the isotropy group of $\lambda\rightarrow\mu$, we recall that there is at most one morphism between any two objects of~$\Lcal_{n}$. Hence $\lambda\rightarrow\mu$ is fixed by $u\in\Un$ if and only if $u$ fixes
both $\lambda$ and $\mu$.
\end{proof}

\section{$\Lcal_{n}$ is simply connected}
\label{section: simple connectivity}

In~\cite{Banff1}, we proved that the nerve of $\Lcal_{3}$ is simply
connected by exhibiting the object space (which has two connected components)
and the morphism space (which has one connected component), and using the Van
Kampen Theorem.  In this section, we use similar methods to show more generally that $\Lcal_{n}$ is simply connected when
$n\geq 3$. Since $\Lcal_{1}$ is empty, and we know
$\Lcal_{2}\cong\reals P^{2}$ \cite[Prop. 2.1]{Banff1}, the following
theorem completes the understanding of the fundamental group of
$\Lcal_{n}$ for all~$n$.

\begin{theorem}  \label{theorem: simply connected}
\simplyconnectedtext
\end{theorem}

Given Theorem~\ref{theorem: simply connected}, the contractibility of
$\Lcal_{n}$ is determined by homology.

\begin{corollary}   \label{corollary: contractible}
\corollarycontractibletext
\end{corollary}

\begin{proof}
The mod~$p$ homology of $\Lcal_{p^{k}}$ is nonzero
\cite[Theorem~4(b)]{Arone-Topology}, so certainly $\Lcal_{p^{k}}$ is not
contractible. If $n$ is not a prime power, then $n>2$, so
$\Lcal_{n}$ is simply connected by
Theorem~\ref{theorem: simply connected}. But
if $n$ is not a power of a prime, then $\Lcal_{n}$ is
acyclic both rationally and at all primes by another result of
Arone (see, for example, \cite[Prop.~9.6]{Arone-Lesh-Crelle}).
The result follows.
\end{proof}

We build on the results of Section~\ref{section: exploration}
to give an elementary approach to proving
Theorem~\ref{theorem: simply connected}.
We begin by considering the fundamental group of each connected
component of~$\Obj\left(\Lcal_{n}\right)$.

\begin{lemma}    \label{lemma: fund group obj component}
Suppose that $\lambda\in\Obj\left(\Lcal_{n}\right)$ has type
$k_{1}\nund_{1}+\dots+k_{j}\nund_{j}$. Then
$\pi_{1}\objcomponent{\lambda}\cong \prod_{i=1}^{j}\Sigma_{k_{i}}$.
\end{lemma}

\begin{proof}
Recall from Lemmas~\ref{lemma: object component homogeneous}
and~\ref{lemma: object isotropy} that
$\objcomponent{\lambda}\cong \Un/\UnIsotropyof{\lambda}$ and
$\UnIsotropyof{\lambda}\cong\prod_{i=1}^{j}\left(U(n_i)\wr\Sigma_{k_i}\right)$
However, if $G$ is a compact Lie group with maximal torus~$T$, then
$\pi_{1}T\rightarrow\pi_{1}G$ is surjective. (See, for example,
Corollary~5.17 in \cite{MimuraToda}.)  The subgroup
$\UnIsotropyof{\lambda}$ contains a maximal torus of~$\Un$, so
$\pi_{1}\UnIsotropyof{\lambda}\rightarrow\pi_{1}\Un$ is an
epimorphism. The lemma then follows from the long exact sequence of homotopy groups associated to
$\UnIsotropyof{\lambda}\rightarrow\Un\rightarrow\Un/\UnIsotropyof{\lambda}$,
because
$\pi_{0}\UnIsotropyof{\lambda}\cong \prod_{i=1}^{j}\Sigma_{k_{i}}$ and
$\Un$ is connected.
\end{proof}

We need an easily-applied criterion
for the Van Kampen Theorem to yield simple connectivity, which is
the purpose of the following algebraic lemma.

\begin{lemma}  \label{lemma: free product zero}
Consider a diagram of groups
\[
A \xleftarrow{\ f\ } C \xrightarrow{\ g\ } B.
\]
Let $K=\ker(C \xrightarrow{\ g\ } B)$ and assume the following conditions.
\begin{enumerate}
\item The normal closure of $\im(K\xrightarrow{\ f\ } A)\subseteq A$ is $A$.
\item The normal closure of $\im(C\xrightarrow{\ g\ } B)\subseteq B$ is $B$.
\end{enumerate}
Then $A \ast_C B$ is the trivial group.
\end{lemma}

\begin{proof}
  We first show that any element of $A$, regarded as an element of
  $A \ast_C B$, is actually trivial.  By definition, elements of the
  free product of $A$ and $B$ that have the form $f(c)g(c)^{-1}$ are
  null in the amalgamated product $A \ast_C B$.   Since the normal closure of $K$ in $A$ is $A$ itself, any element of $A$
  can be written as a product of elements of the form $y=af(k)a^{-1}$
  where $a\in A$ and $k\in K$.
  However, $g(k)=e$, so $y=af(k)a^{-1}=af(k)g(k)^{-1}a^{-1}$, and $y$
  becomes null in $A \ast_C B$.

  Likewise, any element of $B$ can be written as the product of elements
  $z=bg(c)b^{-1}$ where $b\in B$ and $c\in C$. We can write
  $z= bg(c)b^{-1}= bf(c)\left[f(c)^{-1}g(c)\right]b^{-1}$, that is,
  $z= bf(c)b^{-1}$ in the amalgamated product. But there exist
  $k\in K$ and $a\in A$ such that $f(c)=af(k)a^{-1}$. Hence, in the
  amalgamated product $z=b\left(a f(k)a^{-1}\right)b^{-1}$. But since $k \in K$, we know that $f(k)=e$ in the amalgamated product, from which it follows that $z$ also becomes trivial in $A \ast_C B$.
\end{proof}

Let $\skel{i}{\Lcal_{n}}$ denote the $i$-skeleton of the nerve of~$\Lcal_{n}$.
It is sufficient to show that the map
\[
\pi_{1}\skel{1}{\Lcal_{n}}\rightarrow\pi_{1}\skel{2}{\Lcal_{n}}
\]
is trivial, so we begin by analyzing paths in $\skel{1}{\Lcal_{n}}$ in
detail.

Given a decomposition $\lambda\in\Obj\left(\Lcal_{n}\right)$, recall
that $\objcomponent{\typeof{\lambda}}$ is the connected component of
$\Obj\left(\Lcal_{n}\right)$ consisting of all objects with the same
type as~$\lambda$
(Definition~\ref{defn: type objects}~\eqref{item: set with same type} and
Lemma~\ref{lemma: object component homogeneous}).  Likewise,
$\morphcomponent{\lambda}{\mu}$ is the connected component of
$\Morph\left(\Lcal_{n}\right)$ containing morphisms of the same type as~$\lambda\rightarrow\mu$
(Definition~\ref{defn: type morphisms}~\eqref{item: set with same type} and Lemma~\ref{lemma: type=component}). Finally, we write
$\Cyl{\lambda}{\mu}$ for the homotopy pushout of the diagram
\begin{equation}  \label{eq: cylinder}
\xymatrix{
\morphcomponent{\lambda}{\mu}\ar^{\rm{\hspace{1.5em}target\hspace{.5em}}}[r]
       \ar_{\rm{source}}[d] & \objcomponent{\mu}\\
\objcomponent{\lambda} & . }
\end{equation}
The space $\skel{1}{\Lcal_{n}}$ can be written as the (finite) union
of spaces $\Cyl{\lambda}{\mu}$ where $\lambda\rightarrow\mu$ ranges
over a set of representatives of the connected components of
$\Morph\left(\Lcal_{n}\right)$. Note that these cylinders are not
disjoint. For example, the following diagram shows connected
components of $\Obj\left(\Lcal_{4}\right)$, with arrows between those
that are connected by morphisms:
\begin{equation}  \label{eq: L4 diagram}
\xymatrix{
\objcomponent{1\cdot\underline{1}+1\cdot\underline{3}}
        &  & \objcomponent{2\cdot\underline{2}}\\
& \objcomponent{2\cdot\underline{1}+1\cdot\underline{2}}\ar[ul]
    \ar[ur]\\
. &\objcomponent{4\cdot{\underline{1}}}
      \ar@<1ex>[uul]
      \ar[u]
      \ar@<-1ex>[uur]
  }
\end{equation}
Thus $\skel{1}{\Lcal_{4}}$ is the union of five non-degenerate cylinders, three
of which contain the object
component~$\objcomponent{4\cdot\underline{1}}$.

Using the Van Kampen Theorem, the fundamental group of any double
mapping cylinder $\Cyl{\lambda}{\mu}$ can be expressed as a quotient
of the free product of $\pi_{1}\objcomponent{\lambda}$ and
$\pi_{1}\objcomponent{\mu}$.
We illustrate with a special case that turns out to do much of the
work to prove Theorem~\ref{theorem: simply connected}.

\begin{proposition}    \label{prop: tentacle}
Let $\mu$ be any decomposition
containing at least one subspace of dimension greater than one, and
let $\epsilon$ be a refinement of $\mu$ into lines.
Then $\Cyl{\epsilon}{\mu}$ is simply connected.
\end{proposition}

\begin{proof}
Suppose that $\mu$ has type
  $k_{1}\nund_{1}+\dots+k_{j}\nund_{j}$. Then
\[
\UnIsotropyof{\mu}\cong
\left(\Uof{n_1} \wr \Sigma_{k_1}\right) \times\dots\times
\left(\Uof{n_j} \wr \Sigma_{k_j}\right),
\]
and by Lemma~\ref{lemma: fund group obj component} we have
$\pi_1\objcomponent{\mu}\cong\prod_{i=1}^{j}\Sigma_{k_{i}}$.
Likewise
\[
\UnIsotropyof{\epsilon}\cong \Uof{1} \wr \Sigma_{n}
\]
and by Lemma~\ref{lemma: fund group obj component} we have
$\pi_1\objcomponent{\epsilon}\cong\Sigma_n$.
Finally, $\Un$ acts transitively on $\morphcomponent{\epsilon}{\mu}$, with isotropy
group $\UnIsotropyof{\epsilon}\cap\UnIsotropyof{\mu}$
(Lemma~\ref{lemma: type=component}). We can compute
\[
\UnIsotropyof{\epsilon}\cap\UnIsotropyof{\mu}\cong
\left(U(1) \wr \Sigma_{n_1}) \wr \Sigma_{k_1}\right)
      \times \cdots \times
\left(U(1) \wr \Sigma_{n_j}) \wr \Sigma_{k_j}\right),
\]
and as a result, by a similar argument to
Lemma~\ref{lemma: fund group obj component}, we have
\[
\pi_1\morphcomponent{\epsilon}{\mu} \cong
\left(\Sigma_{n_1} \wr \Sigma_{k_1}\right)
      \times \cdots \times
\left( \Sigma_{n_j} \wr \Sigma_{k_j}\right).
\]

Hence applying $\pi_{1}$ to diagram~\eqref{eq: cylinder} gives us
\[\xymatrix{
\left(\Sigma_{n_1} \wr \Sigma_{k_1}\right)
      \times \cdots \times
\left( \Sigma_{n_j} \wr \Sigma_{k_j}\right)
 \ar[r] \ar[d]
&
\Sigma_{k_1}\times \cdots \times \Sigma_{k_j} \\
\Sigma_n & , }
\]
where the top map uses the projection for each of the wreath products,
and the vertical map is inclusion (recall that $n_1k_1+\dots+k_jn_j=n$).
It now follows from Lemma~\ref{lemma: free product zero}
that the fundamental group of $\Cyl{\epsilon}{\mu}$ is trivial.
\end{proof}

If we apply Proposition~\ref{prop: tentacle} to
diagram~\eqref{eq: L4 diagram}, we see that we have simple
connectivity of each of the three cylinders
containing~$\objcomponent{4\cdot\underline{1}}$. However, it is
possible to make homotopically essential loops in
the $1$-skeleton by making circuits of the triangles in the
diagram, so we focus attention on such triangles.

Given a morphism $\lambda\rightarrow\mu$ in~$\Lcal_{n}$,
choose a refinement $\epsilon$ of $\lambda$ into one-dimensional components,
and consider the ``triangle''
\begin{equation}   \label{diagram: triangle}
\xymatrix{
   &\objcomponent{\typeof{\mu}}\\
  \objcomponent{\typeof{\lambda}} \ar[ur]^{\morphcomponent{\lambda}{\mu}}\\
  & \objcomponent{n\cdot{\underline{1}}} \ar[ul]^{\morphcomponent{\epsilon}{\lambda}} \ar[uu]_{\morphcomponent{\epsilon}{\mu}}.
}
\end{equation}

We define
\[
\morphtriangle{\lambda}{\mu}=
\Cyl{\epsilon}{\mu}\cup\Cyl{\epsilon}{\lambda}\cup\Cyl{\lambda}{\mu},
\]
a path-connected space that depends only on the morphism type of $\lambda\rightarrow\mu$. Consider the following paths:
\begin{equation}    \label{eq: defn alpha}
\begin{cases}
\alpha_{1}(t)=\left(\epsilon\rightarrow\lambda,t\right)&
            \mbox{in }\Cyl{\epsilon}{\lambda}\\
\alpha_{2}(t)=\left(\lambda\rightarrow\mu,t\right)&
            \mbox{in }\Cyl{\lambda}{\mu}\\
\alpha_{3}(t)=\left(\epsilon\rightarrow\mu,t\right)&
            \mbox{in }\Cyl{\epsilon}{\mu}
\end{cases}
\end{equation}
Let $\alpha=\alpha_{1}*\alpha_{2}*\alpha_{3}^{-1}$ be the concatenation of the three paths to form a loop in~$\morphtriangle{\lambda}{\mu}$ based at $\epsilon$.

\begin{lemma} \label{lemma: alpha generates fundamental group}
If $\lambda$ contains at least one subspace of dimension greater than one, then
$\pi_{1}\left(\morphtriangle{\lambda}{\mu},\epsilon\right)$ is generated by the loop $\alpha$ given above.
\end{lemma}

\begin{proof}

 Any loop based at $\epsilon$ can be written as a finite concatenation of paths in $\Cyl{\epsilon}{\lambda}$, $\Cyl{\lambda}{\mu}$, and $\Cyl{\epsilon}{\mu}$. Furthermore, path-connectedness of  $\objcomponent{\typeof{\lambda}}$, $ \objcomponent{\typeof{\mu}}$, and $\objcomponent{n\cdot{\underline{1}}}$ means that we can assume that the paths in the concatenation have endpoints at $\epsilon$, $\lambda$, or $\mu$. Hence it is sufficient to analyze paths in the cylinders with those endpoints.

  By Proposition~\ref{prop: tentacle}, $\Cyl{\epsilon}{\lambda}$ is simply connected. Hence any path in $\Cyl{\epsilon}{\lambda}$ beginning at $\epsilon$ and ending at $\lambda$ is path homotopic to~$\alpha_{1}$. Similarly, any path in $\Cyl{\epsilon}{\mu}$ beginning at $\epsilon$ and ending at $\mu$ is path homotopic to $\alpha_{3}$.

  To finish the proof, we need to show that any path in $\Cyl{\lambda}{\mu}$ from $\lambda$ to $\mu$ is path homotopic in $\morphtriangle{\lambda}{\mu}$ to~$\alpha_{2}$. Note that we need the entire triangle. Unlike the previous cases, it is not generally true that a path from $\lambda$ to $\mu$ in $\Cyl{\lambda}{\mu}$ is path homotopic to $\alpha_{2}$ by a path homotopy that stays inside $\Cyl{\lambda}{\mu}$.

  It is sufficient to prove that the inclusion
  $\Cyl{\lambda}{\mu}\subseteq\morphtriangle{\lambda}{\mu}$ induces the trivial map on fundamental groups.
 Observe that by the Van Kampen Theorem,
  $\pi_{1}\Cyl{\lambda}{\mu}$ is a quotient of the free product
  $\pi_{1}\objcomponent{\typeof{\lambda}}
       \ast\pi_{1}\objcomponent{\typeof{\mu}}$.
  However, simple connectivity of $\Cyl{\epsilon}{\lambda}$ means that $\pi_{1}\objcomponent{\typeof{\lambda}}$ maps trivially to
  $\pi_{1}\morphtriangle{\lambda}{\mu}$. Likewise, simple connectivity of
 $\Cyl{\epsilon}{\mu}$
 means that $\pi_{1}\objcomponent{\typeof{\mu}}$ maps trivially to
  $\pi_{1}\morphtriangle{\lambda}{\mu}$.

  The proposition follows, because any loop based at $\epsilon$ is path homotopic to a concatenation of the paths $\alpha_{1}$, $\alpha_{2}$, $\alpha_{3}$, and their inverses.
\end{proof}

Having identified a generator for the potentially nontrivial fundamental group of a triangle, we establish that the generator becomes null in the two-skeleton.
(In fact, the fundamental group of $\morphtriangle{\lambda}{\mu}$ is nontrivial for $n>2$ if and only if $\epsilon$, $\lambda$, and $\mu$ all have different types, but we will not need this fact.)

\begin{proposition}     \label{prop: zero map}
Suppose that $\epsilon\rightarrow\lambda\rightarrow\mu$ are composable morphisms
in~$\Lcal_{n}$, where $\epsilon$ has type~$n\cdot\underline{1}$ and $\lambda$ has at least one subspace of dimension greater than one. Then
$\morphtriangle{\lambda}{\mu}\hookrightarrow\skel{2}{\Lcal_{n}}$ induces
the zero map on fundamental groups.
\end{proposition}

\begin{proof}
By Lemma~\ref{lemma: alpha generates fundamental group},
the fundamental group $\pi_1\morphtriangle{\lambda}{\mu}$ is generated by
the loop $\alpha$ in~\eqref{eq: defn alpha}.
But in $\skel{2}{\Lcal_{n}}$, a null-homotopy for $\alpha$ is provided
by the $2$-simplex corresponding to the composable morphisms
$\epsilon\rightarrow\lambda\rightarrow\mu$, thus proving the proposition.
\end{proof}

By breaking a loop up as a concatenation of loops in triangles, we obtain the desired simple connectivity result.

\begin{proof}[Proof of Theorem~\ref{theorem: simply connected}]
  We can write $\skel{1}{\Lcal_{n}}$
  as the finite union of the path connected spaces
  $\morphtriangle{\lambda}{\mu}$. Therefore any path $\alpha$ based at $\epsilon\in\objcomponent{n\cdot{\underline{1}}}$ can be written as a finite concatenation $\alpha_{1}\ast\dots\ast\alpha_{k}$, where each $\alpha_{i}$ is a path in a triangle $\morphtriangle{\lambda_{i}}{\mu_{i}}$. However, the intersection of neighboring triangles
  $\morphtriangle{\lambda_{i}}{\mu_{i}}\cap\morphtriangle{\lambda_{i+1}}{\mu_{i+1}}$ is path connected and contains~$\objcomponent{n\cdot{\underline{1}}}$, so $\alpha_{i}(1)=\alpha_{i+1}(0)$ can be connected to the basepoint $\epsilon$ by a path $\beta_{i}$ within $\morphtriangle{\lambda_{i}}{\mu_{i}}\cap\morphtriangle{\lambda_{i+1}}{\mu_{i+1}}$. Then $\alpha$ is path homotopic to
\[
  \left(\alpha_{1}\ast\beta_{1}\right)
       \ast\left(\beta_{1}^{-1}\ast\alpha_{2}\ast\beta_{2}\right)
       \ast\dots
       \ast\left(\beta_{k-1}^{-1}\ast\alpha_{k}\right),
\]
which is a concatenation of loops at $\epsilon$ completely contained within a triangles.

  But by Proposition~\ref{prop: zero map},
  any loop in   $\morphtriangle{\lambda}{\mu}$ maps to zero in
  $\pi_1 \skel{2}{\Lcal_{n}}$, which proves the theorem.
\end{proof}

\section{A contractibility criterion for fixed points}
\label{section: normal subgroup condition}

Recall that we call a subgroup of $\Un$ \defining{polytypic} if its action on $\Cn$ is not isotypic.
This section establishes that subgroups $H\subseteq
\Un$ with certain normal, polytypic subgroups have contractible fixed
point sets on $\Lcal_{n}$. The main result is
Proposition~\ref{proposition: polytypic gives contractible} below,
which first appeared as Theorem~4.5 in \cite{Banff1};
its proof is given at the end of the section.
Our goal in
Section~\ref{section: find normal subgroup} will be to find normal
subgroups $J$ satisfying the hypotheses of
Proposition~\ref{proposition: polytypic gives contractible}.

\begin{proposition}
\label{proposition: polytypic gives contractible}
Let $H$ be a subgroup of $\Un$, and suppose $H$ has a normal subgroup
$J$ with the following properties:
\begin{enumerate}
     \item $J$ is polytypic.
     \item For every $\lambda\in \weakfixed{H}$, the action of $J$ on
         $\class(\lambda)$ is not transitive.
\end{enumerate}
Then $\weakfixed{H}$ is contractible.
\end{proposition}

To prove Proposition~\ref{proposition: polytypic gives contractible},
we need two constructions that will give natural retractions
between various subcategories of $\Lcal_{n}$.

\begin{definition}\label{orbit_partition_def}
  Let $J \subseteq \Un$ be a subgroup, and let
  $\lambda$ be a weakly $J$-fixed decomposition of $\Cn$.
\begin{enumerate}
\item The decomposition $\glom{\lambda}{J}$ is defined as follows:
  $w\subseteq\Cn$ is a component of $\glom{\lambda}{J}$
  if and only if $w=\Sigma_{j\in J} \,jv$ for some $v\in\class(\lambda)$.
  That is, a component of $\glom{\lambda}{J}$ is formed by
  taking the sum of all components that are in the same orbit of the action of $J$ on the set of components of~$\lambda$.
\item If $\lambda$ is strongly $J$-fixed, let
  $\isorefine{\lambda}{J}$ be the refinement of $\lambda$ obtained by
  breaking each component of $\lambda$ into its canonical
  $J$-isotypical summands.
\end{enumerate}
\end{definition}

A routine check establishes the following properties of the operations
above.

\begin{lemma}\hfill
\begin{enumerate}
\item The decomposition $\glom{\lambda}{J}$ is strongly $J$-fixed, natural in
$\lambda$, and minimal among coarsenings of $\lambda$ that are
strongly $J$-fixed.
\item
The decomposition $\isorefine{\lambda}{J}$ is natural in strongly $J$-fixed
decompositions $\lambda$ and is maximal among refinements of $\lambda$ whose
classes are isotypical representations of $J$.
\end{enumerate}
\end{lemma}

The following two lemmas give criteria for $\glom{\lambda}{J}$ and
$\isorefine{\lambda}{J}$ to be weakly fixed by a supergroup
$H\subseteq\Un$ of~$J$.  Lemma~\ref{lemma: glom} is a straightforward
check, and Lemma~\ref{lemma: isorefine} involves some basic
representation theory.

\begin{lemma} \label{lemma: glom}
Let $J\triangleleft H$, and suppose
  that $\lambda$ is weakly fixed by $J$.  If $h\in H$, then
  $\glom{(h\lambda)}{J}=h\left(\glom{\lambda}{J}\right)$. In particular, if $\lambda$ is weakly fixed by~$H$, then $\glom{\lambda}{J}$ is also weakly fixed by~$H$.
\end{lemma}

\begin{proof}
We show that a component of $h\left(\glom{\lambda}{J}\right)$ is in fact a component of $\glom{(h\lambda)}{J}$.
Let $v\in\class(\lambda)$, and consider the component $w\in\class\left(\glom{\lambda}{J}\right)$ given by $w=\Sigma_{j\in J}\,jv$. We can compute $hw$ as
\begin{align*}
hw                   &=h\left(\Sigma_{j\in J}\,jv\right) \\
                     &= \Sigma_{j\in J}\,\left(hjh^{-1}\right)(hv)\\
                     &= \Sigma_{j'\in J}\ j'(hv).
\end{align*}
Here $j'$ runs over all elements of $J$ because conjugation by $H$ is an automorphism. Therefore $hw$ is a component in $\glom{(h\lambda)}{J}$.
\end{proof}

\begin{lemma} \label{lemma: isorefine}
  Let $J\triangleleft H$, and suppose that $\lambda$ is strongly fixed
  by $J$. If $h\in H$, then
  $\isorefine{(h\lambda)}{J}=h\left(\isorefine{\lambda}{J}\right)$.
  In particular, if $\lambda$ is weakly fixed by~$H$, then $\isorefine{\lambda}{J}$ is also weakly fixed by~$H$.
\end{lemma}

\begin{proof}
  Let $v\in\class(\lambda)$, and let $h\in H$. Because $\lambda$ is strongly $J$-fixed, we know that $v$ is stabilized by~$J$. The subspace $hv$ is a component of~$h\lambda$, and because $J\triangleleft H$, we can check that $hv$ is also stabilized by~$J$. Hence $v\in\class(\lambda)$ and $hv\in\class(h\lambda)$ are both $J$-representations.

  But as a representation of $J$, the subspace $hv$ is conjugate to the
  representation $v$ of $J$. Thus if $w\subseteq v$ is the
  isotypical summand of $v$ for an irreducible representation $\rho$
  of $J$, then $hw$ is the isotypical summand of $hv$ for the
  conjugate of $\rho$ by $h$. We conclude that $h$ maps the canonical
  $J$-isotypical summands of $v$ to the canonical $J$-isotypical
  summands of $hv$, which are components of~$\isorefine{(h\lambda)}{J}$.
  Since this result is true for every $v\in\class(\lambda)$, we find that
  $\isorefine{(h\lambda)}{J}=h\left(\isorefine{\lambda}{J}\right)$.
\end{proof}

For $J\triangleleft H$, the constructions $\lambda\mapsto\glom{\lambda}{J}$ and
$\lambda\mapsto\isorefine{\lambda}{J}$
will allow us to retract $\weakfixed{H}$ onto subcategories
that in many cases have a terminal object and hence contractible nerve.
As a first step, we need to verify continuity of the constructions.
The proofs use an explicit identification of the path components
of~$\weakfixed{J}$. We need the following lemma and its corollary,
which we learned from S. Costenoble.

\begin{lemma}{\cite[Lemma~1.1]{May-Tel-Aviv}}  \label{lemma: May}
Let $G$ be a compact Lie group with closed subgroups $J$ and $K$. Let $p:\alpha\rightarrow\beta$ be a $G$-homotopy between $G$-maps
$G/J\longrightarrow G/K$. Then $p$ factors as the composite of $\alpha$ and
a homotopy $c \colon G/J\times I\rightarrow G/J$ such that
$c\left(eJ,t\right)=c_{t}J$, where $c_{0}=e$ and the values $c_{t}$ specify
a path in the identity component of the centralizer
$\Cen_{G}(J)$ of $J$ in~$G$.
\end{lemma}

\begin{corollary} \label{cor: J fixed path components}
Let $J\subseteq\Un$ be a closed subgroup, and let $\CentIdent{J}$ denote the identity component of the centralizer of $J$ in~$\Un$.
\begin{enumerate}
\item \label{item: path components}
    The path components of $\Obj\weakfixed{J}$ are $\CentIdent{J}$-orbits.
\item \label{item: equivariant sufficient}
    Any $\CentIdent{J}$-equivariant map from $\Obj\weakfixed{J}$ to itself is continuous.
\end{enumerate}
\end{corollary}

\begin{proof}
The space $\Obj\left(\Lcal_{n}\right)$ is topologized as the disjoint union of $\Un$-orbits $\Un/\UnIsotropyof{\lambda}$, so we can apply Lemma~\ref{lemma: May} with $G=\Un$ to each path component of $\Obj\left(\Lcal_{n}\right)$ to obtain the first result. The second result then follows from the fact that $\Obj\weakfixed{J}$ is topologized as the disjoint union of $\CentIdent{J}$-orbits.
\end{proof}

As a consequence, we obtain continuity of $\lambda\mapsto\glom{\lambda}{J}$.
\begin{lemma}   \label{lemma: continuity of glom}
If $J\triangleleft H$, then
the function $\lambda\mapsto\glom{\lambda}{J}$ is continuous
on~$\Obj\weakfixed{H}$.
\end{lemma}

\begin{proof}
Since $\Obj\weakfixed{H}$ is a subspace of $\Obj\weakfixed{J}$, it suffices to show the lemma for $H=J$.
We need only verify that the hypothesis of Corollary~\ref{cor: J fixed path components}~\eqref{item: equivariant sufficient} holds, that is, that the assignment $\lambda\mapsto\glom{\lambda}{J}$ is $\CentIdent{J}$-equivariant.
However, Lemma~\ref{lemma: glom} tells us the stronger condition that $\lambda\mapsto\glom{\lambda}{J}$ is actually equivariant with respect to the normalizer of~$J$, hence necessarily with respect to $\CentIdent{J}$ as well.
%
\end{proof}

We handle isotypical refinement in a similar way.

\begin{lemma}   \label{lemma: continuity of isorefine}
If $J\triangleleft H$, then
the function $\lambda\mapsto\isorefine{\lambda}{J}$ is continuous
on~$\Obj\strongfixed{H}$.
\end{lemma}

\begin{proof}
Since $\Obj\weakfixed{H}$ is a subspace of $\Obj\weakfixed{J}$, it suffices to show the lemma for $H=J$.
First we observe that $\Obj\strongfixed{J}$ and $\Obj\isofixed{J}$ are both stabilized by the action of~$\CentIdent{J}$, and hence both are unions of path components of~$\Obj\weakfixed{J}$. As in the previous lemma, the result follows from
Corollary~\ref{cor: J fixed path components}~\eqref{item: equivariant sufficient}.
\end{proof}

With continuity established, we can present the key retraction result.

\begin{proposition} \label{proposition: ContractJ}
  Let $H$ be a subgroup of $\Un$, and suppose $J\triangleleft H$. Then
  the inclusion functor
\[
\iota_{1}\colon \isofixed{J}\ \cap\ \weakfixed{H}
\longrightarrow \strongfixed{J}\ \cap\ \weakfixed{H}
\]
induces a homotopy equivalence on nerves. If $J$ further has the
property that $\glom{\lambda}{J}$ is proper for every $\lambda\in
\weakfixed{H}$, then the inclusion functor
\[
\iota_{2}\colon\strongfixed{J}\ \cap\ \weakfixed{H}\longrightarrow \weakfixed{H}
\]
induces a homotopy equivalence on nerves.
\end{proposition}

\begin{proof}
  By Lemmas~\ref{lemma: isorefine} and~\ref{lemma: continuity of isorefine},
  a continuous, functorial retraction of $\iota_{1}$ is given by
  $r_{1}\colon \lambda\mapsto\isorefine{\lambda}{J}$. The coarsening
  morphism
  $\isorefine{\lambda}{J}\rightarrow\lambda$
  provides a natural transformation from $\iota_{1}r_{1}$ to the
  identity, establishing the first desired equivalence.

  Similarly, by Lemmas~\ref{lemma: glom} and~\ref{lemma: continuity of glom}, a continuous, functorial retraction of
  $\iota_{2}$ is given by $r_{2}\colon\lambda\mapsto\glom{\lambda}{J}$,
  because by assumption $\glom{\lambda}{J}$ is always a proper decomposition of
  $\Cn$. The coarsening morphism
  $\lambda\rightarrow\glom{\lambda}{J}$ provides a natural
  transformation from the identity to $\iota_{2}r_{2}$, which
  establishes the second desired equivalence.
\end{proof}

\begin{proof}
[Proof of Proposition~\ref{proposition: polytypic gives contractible}]
Let $\mu$ denote the decomposition of $\Cn$ into the
canonical isotypical components of $J$.  If $J$ is polytypic, then
$\mu$ has more than one component and thus is proper, and $\mu$ is
terminal in~$\isofixed{J}$.

We assert that $\mu$ is a terminal object in $\isofixed{J}\ \cap\
\weakfixed{H}$, and we need only establish that $\mu$ is weakly
$H$-fixed.  But $\mu$ is the $J$-isotypical refinement of the
indiscrete decomposition of $\Cn$, i.e., the decomposition
consisting of just $\Cn$ itself. Since the indiscrete
decomposition is certainly $H$-fixed, Lemma~\ref{lemma: isorefine} tells
us that $\mu$ is weakly $H$-fixed.

The result now follows from
Proposition~\ref{proposition: ContractJ}, because the assumption that
the action of $J$ on $\class(\lambda)$ is always intransitive means
that $\glom{\lambda}{J}$ is always proper.
\end{proof}

\section{Finding a normal subgroup}
\label{section: find normal subgroup}

Throughout this section, suppose that $H\subseteq \Un$ is a \pdash toral
subgroup of~$\Un$, and let $\Hbar$ denote the image of $H$ in~$P\Un$. Note that
although $P\Un$ does not act on $\Cn$, it does act on
$\Lcal_{n}$, because the central $S^{1}\subseteq\Un$
stabilizes any subspace of $\Cn$. Hence, for example,
if a decomposition $\lambda$ is weakly $H$-fixed, we can speak of
the action of $\Hbar$ on~$\class(\lambda)$.

Our goal is to prove Theorem~\ref{thm: H elementary abelian}, which
says that if $H$ is problematic, then $\Hbar$ is elementary abelian
and $\character{H}$ is zero except on elements that are
in the central $S^{1}$ of~$\Un$. The plan for the proof is to apply
Proposition~\ref{proposition: polytypic gives contractible} after
locating a relevant normal polytypic subgroup of~$H$.
The following
lemma gives us a starting point, and the proof of
Theorem~\ref{thm: H elementary abelian} appears at the end of the section.

\begin{lemma}    \label{lemma: abelian means polytypic}
If $J\subseteq \Un$ is abelian, then either $J$ is polytypic or
$J\subseteq S^{1}$.
\end{lemma}

\begin{proof}
Decompose $\Cn$ into a sum of $J$-irreducible representations,
all of which are necessarily one-dimensional because $J$ is abelian.
If $J$ is isotypic, then an element $j\in J$ acts on every
one-dimensional summand by multiplication by the same scalar, so
$j\in S^{1}$.
\end{proof}

Lemma~\ref{lemma: abelian means polytypic} places an immediate
restriction on problematic subgroups.

\begin{lemma} \label{lemma: only rank one torus}
  If $H$ is a problematic \pdash toral subgroup of $\Un$,
  then $\Hbar$ is discrete.
\end{lemma}

\begin{proof}
  We assert that, without loss of generality, we can assume that $H$
  actually contains $S^{1}$.  Because
  $\weakfixed{H}=\weakfixed{HS^{1}}$, if $H$ is problematic, then
  so is~$HS^{1}$. Likewise, $H$ is $p$-toral if and only if
  $HS^{1}$ is \pdash toral. Further, $H$ and $HS^{1}$ have the same image
  in~$P\Un$.
  Hence we can assume that $S^{1}\subseteq H$, by replacing $H$ by $HS^{1}$
  if necessary.

  The group $\Hbar$ is \pdash toral (e.g., \cite{Banff1}, Lemma~3.3), so
  its identity component, denoted by $\Hbar_{0}$, is a torus. Let $J$
  denote the inverse image of $\Hbar_{0}$ in $H$; thus
  $J\triangleleft H$. Since we have a fibration $S^{1}\rightarrow
  J\rightarrow\Hbar_{0}$, we know that $J$ is connected.  Further, $J$
  is a torus, because it is a connected closed subgroup of the
  identity component of $H$, which is a torus.

  If $\Hbar$ is not discrete, then $J$ is not
  contained in $S^{1}$, and thus $J$ is polytypic by
  Lemma~\ref{lemma: abelian means polytypic}.  Since $J$ is
  connected, its action on the set of equivalence classes of any
  proper decomposition is trivial. The lemma follows from
  Proposition~\ref{proposition: polytypic gives contractible}.
\end{proof}

In terms of progress towards Theorem~\ref{thm: H elementary abelian},
we now know that if $H$ is a problematic \pdash toral subgroup of~$\Un$,
then $\Hbar$ must be a finite \pdash group. The next part of our strategy
is to show that if $\Hbar$ is not an elementary abelian \pdash group,
then $\Hbar$ has a normal subgroup $V$ satisfying the conditions of
the following lemma.

\begin{lemma} \label{lemma: detection mechanism}
  Let $H\subseteq \Un$, and assume there exists
  $V\triangleleft\Hbar$ such that $V\cong\integers/p$
  and $V$ does not act transitively on $\class(\lambda)$ for any
  $\lambda\in\weakfixed{H}$. Then $\weakfixed{H}$ is contractible.
\end{lemma}

\begin{proof}
  Let $J$ be the inverse image of $V$ in $H$. Then $J\triangleleft H$,
  and because the action of $J$ on $\Lcal_{n}$ factors through $V$,
  the action of $J$ on $\class(\lambda)$ is not transitive for any
  $\lambda\in\weakfixed{H}$.  Further, $J$ is abelian, because a
  routine splitting argument shows $J\cong V\times S^{1}$
  (see~\cite{Banff1}, Lemma~3.1).  Therefore $J$ is polytypic by
  Lemma~\ref{lemma: abelian means polytypic}, and the lemma follows
  from Proposition~\ref{proposition: polytypic gives contractible}.
\end{proof}

Before we prove our first theorem, we recall that if $G$ is a finite
\pdash group, then its Frattini subgroup, $\Phi(G)$, is generated by
the commutators $[G,G]$ and the \pdash fold powers $G^p$. It has the
property that $G\rightarrow G/\Phi(G)$ is initial among maps from $G$
to elementary abelian \pdash groups.

We now have all the ingredients we require to prove
Theorem~\ref{thm: H elementary abelian}.
We note that both the statement and the proof
are closely related to those of Proposition~6.1 of~\cite{ADL2}.

\begin{theorem} \label{thm: H elementary abelian}
  \elementaryabeliantheoremtext
\end{theorem}

\begin{proof}
For~\eqref{item: elem abelian},
  we must show that $\Hbar$ is an elementary abelian \pdash group.
  By Lemma~\ref{lemma: only rank one torus}, we know that $\Hbar$ is a
  finite \pdash group.  If $\Hbar$ is not elementary abelian, then the
  kernel of $\Hbar\rightarrow\Hbarmodp$ is a nontrivial normal
  \pdash subgroup of $\Hbar$, and thus has nontrivial intersection with
  the center of $\Hbar$.  Choose $V\cong\integers/p$ with
\[
V\subseteq
    \ker\left[\,\Hbar\rightarrow\,\Hbarmodp\right]
           \cap Z\left(\Hbar\right).
\]

We assert that $V$ satisfies the conditions of
Lemma~\ref{lemma: detection mechanism}. Certainly
$V\triangleleft\Hbar$, because $V$ is contained in the center
of~$\Hbar$.  Suppose that there exists $\lambda\in\weakfixed{H}$ such
that $V$ acts transitively on $\class(\lambda)$. The set
$\class(\lambda)$ must then have exactly $p$ elements. The action of
$\Hbar$ on $\class(\lambda)$ induces a map
$\Hbar\rightarrow\Sigma_{p}$, and since $\Hbar$ is a \pdash group, the
image of this map must lie in a Sylow \pdash subgroup of~$\Sigma_{p}$.
But then the map $\Hbar\rightarrow\Sigma_{p}$ factors as
\[
\Hbar\rightarrow\Hbarmodp\rightarrow\integers/p\hookrightarrow\Sigma_{p},
\]
with $V\subseteq\Hbar$ mapping nontrivially to
$\integers/p\subseteq\Sigma_{p}$.  We have thus contradicted the
assumption that $V\subseteq\ker\left[\Hbar\rightarrow\Hbarmodp\right]$.
We conclude that $\Hbar$ is, in fact, an elementary abelian \pdash
group, and so \eqref{item: elem abelian} is proved.

For~\eqref{item: character}, first note that if $h\in S^{1}$, then its
matrix representation in $\Un$ is $h\Id$, hence
$\character{H}(h)=\tr(h\Id)=nh$.  Suppose that $\Hbar$ is an elementary abelian
\pdash group, and consider an arbitrary element $h\in H$ such that $h\not\in S^{1}$.  The image of
$h$ in $\Hbar$ generates a subgroup $V\cong\integers/p\subseteq\Hbar$,
and $V$ is a candidate for applying
Lemma~\ref{lemma: detection mechanism}. Since $\weakfixed{H}$ is
not contractible, there must be a decomposition $\lambda\in\weakfixed{H}$
such that $V$ acts transitively on~$\class(\lambda)$.  The action of
$V$ on $\class(\lambda)$ is necessarily free because
$V\cong\integers/p$, so a basis for $\Cn$ can be
constructed that is invariant under $V$ and consists of bases for the
subspaces in $\class(\lambda)$. This action represents $h$ as a fixed-point
free permutation of a basis of~$\Cn$. Hence $\character{H}(h)=0$.
\end{proof}

\section{The subgroups $\Gamma_{k}$
       of $\Upk$}
\label{section: Gamma_k}

Theorem~\ref{thm: H elementary abelian} tells us that if $H$ is a
problematic \pdash toral subgroup of $\Un$, then $H$ is a projective elementary
abelian \pdash group and the character of $H$ is zero away from the
center of $\Un$.  In fact, there are well-known subgroups of $\Un$
that satisfy these conditions, namely the subgroups
$\Gamma_{k}\subseteq \Upk$ that arise in, for example,
\cite{Griess}, \cite{JMO}, \cite{Oliver-p-stubborn},
\cite{Arone-Topology}, ~\cite{Arone-Lesh-Crelle}, \cite{AGMV}, and others. In
this section, we review background on the groups~$\Gamma_{k}$.

We begin with the discrete analogue of~$\Gamma_{k}$. Let $n=p^{k}$ and
choose an identification of the elements of $(\integers/p)^{k}$ with
the set $\{1,...,n\}$. The action of $(\integers/p)^{k}$ on itself by
translation identifies $(\integers/p)^{k}$ as a transitive elementary
abelian \pdash subgroup of $\Sigma_{p^{k}}$, denoted by~$\Delta_{k}$.  Up to
conjugacy, $\Delta_{k}$ is the unique transitive elementary abelian
\pdash subgroup of $\Sigma_{p^{k}}$. Note that every nonidentity element of
$\Delta_{k}$ acts without fixed points. The embedding
\[
\Delta_{k}\hookrightarrow\Sigma_{p^{k}}\hookrightarrow\Upk
\]
given by permuting the standard basis elements is the regular
representation of $\Delta_{k}$, and has character $\character{\Delta_{k}}=0$
except at the identity.

In the unitary context, the projective elementary abelian \pdash subgroup
$\Gamma_{k}\subseteq \Upk$ is generated by the central
$S^{1}\subseteq\Upk$ and two different embeddings of $\Delta_{k}$ in
$\Upk$, which we denote by $\Acal_{k}$ and $\Bcal_{k}$ and describe
momentarily. Just as $\Delta_{k}$ is the unique (up to conjugacy)
elementary abelian \pdash subgroup of $\Sigma_{p^k}$ with transitive
action, it turns out that $\Gamma_{k}$ is the unique (up to conjugacy)
projective elementary abelian \pdash subgroup of $\Upk$ containing the
central $S^{1}$ and having irreducible action (see, for example,
\cite{Zolotykh}). For the explicit description of $\Gamma_{k}$, we
follow~\cite{Oliver-p-stubborn}.

The subgroup $\Bcal_{k}\cong\Delta_{k}$ of $\Upk$ is given as follows.
Consider $\integers/p\subseteq U(p)$ acting by the regular representation,
and let $\Bcal_{k}$ be the group $(\integers/p)^{k}$ acting on
the $k$-fold tensor power $\left(\complexes^{p}\right)^{\otimes k}$.
(This action is, in fact, the regular representation of
$\Delta_{k}\cong (\integers/p)^{k}$.)
Explicitly, for any $r = 0,1,
\ldots, k-1$, let $\sigma_r \in \Sigma_{p^k}$ denote the permutation
defined by
\[
\sigma_r (i) =
\begin{cases}
i + p^r      &\mbox{if } i \equiv 1, \ldots, (p-1)p^r \qquad\quad\mod p^{r+1}, \\
i - (p-1)p^r & \mbox{if } i \equiv (p-1)p^r+1, \ldots, p^{r+1} \mod p^{r+1}.
\end{cases}
\]
For each $r$, let $B_r \in\Upk$ be the corresponding permutation
matrix,
\[
\left(B_r\right)_{ij} =
\begin{cases} 1 &\mbox{if } \sigma_r(i) = j\\ 0 &\mbox{if }
\sigma_r(i) \neq j.
\end{cases}
\]
For later purposes, we record the following lemma.

\begin{lemma}   \label{lemma: character B_k zero}
The character $\character{\Bcal_{k}}$ is zero except at the identity.
\end{lemma}

\begin{proof}
  Every nonidentity element of $\Bcal_{k}$ acts by a fixed-point free
  permutation of the standard basis of $\Cpk$, and so
  has only zeroes on the diagonal.
\end{proof}

Our goal is to define $\Gamma_{k}$ with irreducible action on
$\Cpk$, but $\Bcal_{k}$ alone does not act
irreducibly: being abelian, the subgroup $\Bcal_{k}$ has only
one-dimensional irreducibles. In fact, since $\Bcal_{k}$ is acting on
$\Cpk$ by the regular representation, each of its
$p^{k}$ irreducible representations is present exactly once. The role
of the other subgroup $\Acal_{k}\cong\Delta_{k}\subseteq\Gamma_{k}$ is
to permute the irreducible representations of $\Bcal_{k}$.  To be
specific, let $\zeta = e^{2 \pi i/p}$, and consider
$\integers/p\subseteq \Upp$ generated by the diagonal matrix with
entries $1,\zeta,\zeta^2,...,\zeta^{p-1}$. Then
$\Acal_{k}\subseteq\Upk$ is the group $\left(\integers/p\right)^{k}$
acting on the $k$-fold tensor power
$\left(\complexes^{p}\right)^{\otimes k}$.  Explicitly, for $r = 0,
\ldots, k-1$ define $A_{r}\in\Upk$ by
\[
(A_r)_{ij} =
\begin{cases}
     \zeta^{\left[ (i-1)/p^r\right]} & \mbox{if } i=j \\
     0                    &\mbox{if }  i\neq j
\end{cases}
\]
where $\left[\,\whatever\,\right]$ denotes the greatest integer function.
The matrices $A_{0}$,...,$A_{k-1}$ commute, are of order~$p$,
and generate a rank~$k$ elementary abelian \pdash group~$\Acal_{k}$.

\begin{lemma}  \label{lemma: character A_k zero}
The character $\character{\Acal_{k}}$ is zero except at the identity.
\end{lemma}

\begin{proof}
  The character of $\integers/p\subseteq \Upp$ generated by the
  diagonal matrix with entries $1,\zeta,\zeta^2,...,\zeta^{p-1}$
  is zero away from the identity by direct computation, because
\[
1+\zeta+\zeta^2+...+\zeta^{p-1}=\frac{\zeta^{p}-1}{\zeta-1}=0.
\]
  The same is true for $\Acal_{k}$, because the character is obtained
  by multiplying together the characters of the individual factors.
\end{proof}

Since the characters $\character{\Acal_{k}}$ and $\character{\Bcal_{k}}$
are the same, we obtain the following corollary.

\begin{corollary}
The subgroups $\Acal_{k}$ and $\Bcal_{k}$ are conjugate in $\Upk$.
\end{corollary}

Although $\Acal_{k}$ and $\Bcal_{k}$ do
not quite commute with each other in~$\Upk$, the commutator
relations are simple and follow from examining the actions of
$\Bcal_{k}\cong(\integers/p)^{k}$ and
$\Acal_{k}\cong(\integers/p)^{k}$ on
$\left(\complexes^{p}\right)^{\otimes k}$.  If $r\neq s$, then $A_r$
and $B_s$ are acting on different tensor factors in
$\left(\complexes^{p}\right)^{\otimes k}$, and hence they commute.
The commutator of $A_{r}$ and $B_{r}$, which both act on the
$r$th tensor factor of $\complexes^{p}$, can be computed by an explicit
computation in $U(p)$. As a result, we obtain the following relations:
\begin{align}
\label{eq: Gamma_k_commutator_equations}
[A_r,A_s] &= \Id = [B_r,B_s], & \text{for all } r,s \nonumber\\
[A_r,B_s] &= \Id,             & \text{for all } r\neq s\\
[B_r,A_r] &= \zeta \Id,       & \text{for all } r\nonumber.
\end{align}

\begin{definition}
The subgroup $\Gamma_{k}\subseteq\Upk$ is generated by the subgroups
$\Acal_{k}$, $\Bcal_{k}$, and the central $S^{1}\subseteq\Upk$.
\end{definition}

\begin{lemma}
\label{lemma: Gamma_k extension}
There is a short exact sequence
\begin{equation}
1\rightarrow S^{1}\rightarrow\Gamma_{k}
                  \rightarrow \left(\Delta_{k}\times\Delta_{k}\right)
                  \rightarrow 1.
\end{equation}
where $S^{1}$ is the center of $\Upk$.
\end{lemma}

\begin{proof}
  The subgroups $\Acal_{k}$ and $\Bcal_{k}$ can be taken as the preimages of the
  two copies of $\Delta_{k}$ in $\Gamma_{k}$.  The commutator
  relations \eqref{eq: Gamma_k_commutator_equations} show that
  $\Acal_{k}$ and $\Bcal_{k}$ do not generate any noncentral elements.
\end{proof}

\begin{remark*}
When $k=0$ we have $\Gamma_{0}=S^{1}\subseteq\Uof{1}$, and $\Delta_{0}$
is trivial, so Lemma~\ref{lemma: Gamma_k extension} is true even
for $k=0$.
\end{remark*}

For later purposes, we record the following lemma.

\begin{lemma}  \label{lemma: Gamma_k subgroups}
The subgroup $\Gamma_{k}\subseteq\Upk$ contains subgroups isomorphic
to $\Gamma_{s}\times\Delta_{t}$ for all nonnegative integers
$s$ and $t$ such that $s+t\leq k$.
\end{lemma}

\begin{proof}
The required subgroup is generated by
$S^{1}$, the matrices $A_{0}$, \dots, $A_{s+t-1}$, and
the matrices $B_{0}$, \dots,~$B_{s-1}$.
\end{proof}

A consequence of Lemma~\ref{lemma: Gamma_k extension} is that
$\Gamma_{k}\subseteq\Upk$ is an example of a \pdash toral subgroup that
satisfies the first conclusion of
Theorem~\ref{thm: H elementary abelian}. The next lemma says that
$\Gamma_{k}$ satisfies the second conclusion as well.

\begin{lemma}     \label{lemma: Gamma zero character}
The character of $\Gamma_{k}$ is nonzero only on the elements of~$S^{1}$.
\end{lemma}

\begin{proof}
  The character of $\Gamma_{k}$ on elements of $\Acal_{k}$ and $\Bcal_{k}$ is zero by
  Lemmas~\ref{lemma: character B_k zero}
  and~\ref{lemma: character A_k zero}. Multiplying any of these
  matrices by an element of $S^{1}$, i.e., a scalar, gives a
  matrix that also has zero trace.  Finally, products of nonidentity elements
  of $\Bcal_{k}$ with elements of $\Acal_{k}S^{1}$ are obtained by
  multiplying matrices in $\Bcal_{k}$, which have only zero entries on
  the diagonal, by diagonal matrices. The resulting products likewise
  have no nonzero diagonal entries and thus have zero trace.
\end{proof}

\begin{remark*}
  We note that, by inspection of the commutator relations,
  $S^{1}\times\Acal_{k}$ and $S^{1}\times\Bcal_{k}$ normalize each
  other in $\Gamma_{k}$. Suppose we decompose $\Cpk$ by
  the $p^{k}$ one-dimensional irreducible representations of
  $\Acal_{k}$.  That decomposition is weakly fixed by $\Bcal_{k}$,
  and further, $\Bcal_{k}$ acts transitively on the classes in the
  decomposition because $\Gamma_{k}$ is irreducible. Likewise, the
  decomposition of $\Cpk$ by the $p^{k}$ one-dimensional
  irreducible representations of $\Bcal_{k}$ is weakly fixed by
  $\Acal_{k}$, which has transitive action on the classes.
\end{remark*}

\section{Alternating forms}
\label{section: alternating forms}

From Theorem~\ref{thm: H elementary abelian}, we know that if $H$ is a
problematic \pdash toral subgroup of $\Un$, then its image $\Hbar$ in
$P\Un$ is an elementary abelian \pdash group.  We would like to know
the possible group isomorphism types of such subgroups of~$\Un$.
For simplicity, we restrict ourselves to subgroups $H$ that
contain the central~$S^{1}$ of~$U(n)$.
(See the proof of Lemma~\ref{lemma: only rank one torus}.)
Our main results are
Propositions~\ref{proposition: forms=groups}
and~\ref{proposition: group classification}, below.
Once the
group-theoretic classification is complete, we use representation
theory in Section~\ref{section: problematic subgroups} to pin down the
conjugacy classes of elementary abelian \pdash subgroups of~$\Un$ that
can be problematic.

Before proceeding, we note that the remarkable paper \cite{AGMV} of
Andersen-Grodal-M{\o}ller-Viruel classifies non-toral elementary
abelian \pdash subgroups (for odd primes $p$) of the simple and
center-free Lie groups. In particular, Theorem~8.5 of that work
contains a classification of all the elementary abelian \pdash
subgroups of $P\Un$, building on earlier work of Griess~\cite{Griess}.
Our approach is independent of this classification, and works for all
primes, using elementary methods.

We make the following definition.

\begin{definition}  \label{defn: abstract proj elem}
A \pdash toral group $H$ is an \definedas{abstract projective
elementary abelian \pdash group} if $H$ can be written as a central
extension
\[
1\rightarrow S^{1}\rightarrow H\rightarrow V\rightarrow 1,
\]
where $V$ is an elementary abelian \pdash group.
\end{definition}

We begin by recalling some background on forms.  Let $A$ be a
finite-dimensional $\field_p$-vector space, and let
$\alpha\colon A\times A \to \field_p$ be a bilinear form.  We say that
$\alpha$ is \definedas{totally isotropic} if $\alpha(a,a)=0$ for all
$a\in A$. (A totally isotropic form is necessarily skew-symmetric, as
seen by expanding $\alpha(a+b, a+b)=0$, but the reverse is not true for
$p=2$.)  If $\alpha$ is not only
totally isotropic, but also non-degenerate, then it is called a
\definedas{symplectic form}. Any vector space with a symplectic form
is even-dimensional and has a (nonunique) basis
$e_1,\dots,e_s,f_1,\dots, f_s$, called a \definedas{symplectic basis},
with the property that $\alpha(e_i,e_j)=0=\alpha(f_i,f_j)$ for all
$i,j$, and $\alpha(e_i,f_j)$ is $1$ if $i=j$ and zero otherwise. All
symplectic vector spaces of the same dimension are isomorphic (i.e.,
there exists a linear isomorphism that preserves the form), and if the
vector space has dimension $2s$ we use $\HH_s$ to denote the
associated isomorphism class of symplectic vector spaces.  Let $\TT_t$
denote the vector space of dimension $t$ over $\field_{p}$ with
trivial form.  We have the following standard classification
result.

\begin{lemma}  \label{lemma: classification of forms}
Let $A$ be a vector space over $\field_{p}$ with a
totally isotropic bilinear form $\alpha$. Then there exist $s$ and $t$
such that $A\cong\HH_{s}\oplus\TT_{t}$ by a form-preserving
isomorphism.
\end{lemma}

Our next task is to relate the preceding discussion to abstract
projective elementary abelian \pdash groups. For the remainder of the
section, assume that $H$ is an abstract projective elementary abelian \pdash
group
\begin{equation*}
1\rightarrow S^{1}\rightarrow H\rightarrow V\rightarrow 1.
\end{equation*}
Choose an identification of
$\integers/p$ with the elements of order $p$ in $S^{1}$.
Given $x,y\in V$, let $\xwiggle, \ywiggle$ be lifts of $x,y$ to $H$.
Define the \definedas{commutator form associated to $H$} as the form
on $V$ defined by
\begin{equation}  \label{eq: commutator form}
\alpha(x,y)=
\left[\xwiggle,\ywiggle\right]=\xwiggle\ywiggle\xwiggle^{-1}\ywiggle^{-1}.
\end{equation}

\begin{lemma}   \label{lemma: BilinearForm1}
Let $H$ and $\alpha$ be as above, and suppose that $x,y\in V$. Then
\begin{enumerate}
\item $\alpha(x,y)$ is a well-defined element of
    $\integers/p\subseteq S^{1}$,
\item $\alpha$ is a totally isotropic bilinear form on $V$, and
\item isomorphic groups $H$ and $H'$ give isomorphic forms $\alpha$
and $\alpha'$.
\end{enumerate}
\end{lemma}

\begin{proof}
  Certainly $\left[\xwiggle,\ywiggle\right]\in S^{1}$, since $x$ and
  $y$ commute in $V$. If $\zeta\in S^{1}$ then
  $[\zeta\xwiggle,\ywiggle]=[\xwiggle,\ywiggle]=[\xwiggle,\zeta\ywiggle]$
  because $S^{1}$ is central in $H$, which shows that $\alpha$ is
  independent of the choice of lifts $\tilde x$ and $\tilde y$, and
  that $\alpha$ is linear with respect to scalar multiplication.

  To show that $[\xwiggle,\ywiggle]$ has order $p$, we note that
  commutators in a group satisfy the following versions of the
  Hall-Witt identities, as can be verified by expanding and simplifying:
\begin{align*}
[a,bc] 
       &=[a,b]\cdot \big[b,[a,c]\big]\cdot [a,c]\\
%
[ab,c]
      &= \big[a,[b,c] \big]\cdot [b,c]\cdot [a,c].
\end{align*}
We know that $[H,H]$ is contained in the center of $H$, so $[H,[H,H]]$
is the trivial group. Hence for~$H$, the identities reduce to
 \begin{equation}
\label{eq: Witt Hall}
 \begin{aligned}{}
     [a,bc] &= [a,b][a,c] \\
     [ab,c] &= [a,c][b,c].
 \end{aligned}
 \end{equation}
In particular,
$\left[\xwiggle,\ywiggle\right]^{p}
        =\left[\xwiggle,\ywiggle^{p}\right]= e$,
since $\ywiggle^{p}\in S^{1}$ commutes with $\xwiggle$.
Bilinearity of $\alpha$ with respect to addition follows directly
from~\eqref{eq: Witt Hall}. The form is totally isotropic because
an element commutes with itself.

Finally, an isomorphism $H\rightarrow H'$ necessarily restricts to an
isomorphism on the identity component and induces a diagram
\[
\begin{CD}
1@>>> S^{1}@>>> H@>>> V@>>> 1\\
@. @V{\cong}VV @V{\cong}VV @VVV\\
1@>>> S^{1}@>>> H'@>>> V'@>>> 1,
\end{CD}
\]
which, in turn, induces an isomorphism of the associated forms $\alpha$
and~$\alpha'$ on $V$ and~$V'$, respectively.
\end{proof}

Lemma~\ref{lemma: BilinearForm1} shows
that~\eqref{eq: commutator form} gives a function from isomorphism
classes of projective elementary abelian \pdash groups to isomorphism
classes of totally isotropic forms over~$\integers/p$.  Conversely, we
can start with a form $\alpha$ and directly construct an
abstract projective elementary abelian \pdash group~$H_{\alpha}$. Let
$\alpha\colon V\times V\rightarrow\integers/{p}$ be a totally
isotropic bilinear form, which can be regarded as a function
$\alpha\colon V\times V\rightarrow S^{1}$. Let
\begin{align*}
V_{K} & =\{v\in V \mid \alpha(v,x)=0 \ \text{ for all } x\in V\}\\
  & = \ker\left( V\rightarrow V^{\ast} \right),
\end{align*}
where $V^\ast$ denotes the dual of $V$. Then $\alpha$ restricted to
$V_{K}$ is trivial.

Now we make some choices, and we address the issue of the choices a
little later in the section.  Choose a complement $V_{c}$ to $V_{K}$ in
$V$; note $V_{c}$ is necessarily orthogonal to~$V_{K}$. Since $\alpha$ must
be symplectic on $V_{c}$, we can choose a symplectic basis
$e_{1},\dots,e_{r},f_{1},\dots,f_{r}$ for $V_{c}$; let
$V_{E}$ and $V_{F}$ denote the spans of $E=\{e_{1},\dots,e_{r}\}$ and
$F=\{f_{1},\dots,f_{r}\}$, respectively.  By construction, we can write
any $v\in V$ uniquely as a sum $v=\vvec{K}+\vvec{E}+\vvec{F}$ where
$\vvec{K}\in V_{K}$, $\vvec{E}\in V_{E}$, and $\vvec{F}\in V_{F}$.

Let $H_{\alpha}$ be the set $S^{1}\times V$.
Given the previous choices, we can
endow $H_{\alpha}$ with the following operation~$\alpha(E,F)$:
\begin{equation}  \label{eq: group operation}
(z,v)*_{\alpha(E,F)}(z',v')
     =\left(\strut zz'\,\alpha(\vvec{F},\vprimevec{E}), \ v+v'\right).
\end{equation}

\begin{proposition}\label{prop: form gives extension}
For an elementary abelian \pdash group $V$ and a totally isotropic
bilinear form $\alpha \colon V\times V\to \integers/p$ as above,
we have the following.
\begin{enumerate}
\item The operation~\eqref{eq: group operation} gives $H_{\alpha}$ the
  structure of an abstract projective elementary abelian \pdash group
  with associated commutator pairing~$\alpha$.
\item The group isomorphism class of $H_{\alpha}$ depends only on the
  isomorphism class of $\alpha$ as a bilinear form.
\end{enumerate}
Further, non-isomorphic forms $\alpha$ and $\alpha'$ on $V$ give
nonisomorphic groups $H_{\alpha}$ and $H_{\alpha'}$.
\end{proposition}

\begin{proof}
  The element $(1,0)$ serves as the identity in $H_\alpha$. By
  bilinearity of~$\alpha$, we know
  $\alpha\left(\vvec{F},-\vvec{E}\right)+\alpha\left(\vvec{F},\vvec{E}\right)=0$,
  which allows us to check that the inverse of $(z,v)$ is
  $\left(z^{-1}\alpha\left(\vvec{F},\vvec{E}\right),-v\right)$.  A
  straightforward computation verifies associativity, showing that
  $*_{\alpha(E,F)}$ defines a group law, and another shows that
  $H_{\alpha}$ has $\alpha$ for its commutator pairing.

  We need to check the effect of the choices we made when we
  defined the operation~$\alpha(E,F)$ on~$H_{\alpha}$.  The subspace
  $V_{K}\subseteq V$ is well-defined, but $V_{E}$ and $V_{F}$ are not,
  and they are used in the definition~~$\alpha(E,F)$.
  Suppose that $E,F$ and $E',F'$ are two
  choices for a symplectic basis spanning (not necessarily identical)
  complements of $V_{K}$ in~$V$. There is an automorphism of $\alpha$
  that takes $E$ to $E'$
  and $F$ to $F'$, and then this automorphism defines an isomorphism
$\left(H_{\alpha},\alpha(E,F)\right)
          \cong \left(H_{\alpha},\alpha(E',F')\right)$.
  Hence the group isomorphism class of $H_{\alpha}$ is well-defined,
  independent of the choices made to define the group operation.

Similarly, if $V\xrightarrow{\cong}V'$ induces an isomorphism of
forms $\alpha, \alpha'$, then compatible choices can be made for the
symplectic bases $E,F\subseteq V$ and $E', F'\subseteq V'$, and
these choices will induce an
isomorphism of $\left(H_{\alpha}, \alpha(E,F)\right)$ with
$\left(H_{\alpha'}, \alpha'(E',F')\right)$.

Lastly, if $\alpha$ and $\alpha'$ are not isomorphic, then by
Lemma~\ref{lemma: classification of forms} their trivial components
must be of different dimensions.  It follows that the centers of
$H_{\alpha}$ and $H_{\alpha'}$ are not isomorphic (for example, they
have a different number of connected components), and $H_{\alpha}$ and
$H_{\alpha'}$ are therefore not isomorphic as groups.
\end{proof}

Proposition~\ref{prop: form gives extension} tells us that the construction
$\alpha\mapsto H_{\alpha}$ defines a monomorphism from the
isomorphism classes of totally isotropic bilinear forms over $\integers/p$
to the isomorphism classes of abstract projective elementary abelian
$p$-groups. It remains to show that this function is an epimorphism,
which we do by constructing a group for each form.
For later purposes, we pay special attention to the identity component.

\begin{proposition} \label{proposition: one-to-one}
  Let $H$ be an abstract projective elementary abelian \pdash group
  with associated commutator form
  $\alpha\colon V\times V\rightarrow\integers/p$.  Let
  $\phi:S^{1}\rightarrow H_{0}$ be an isomorphism of $S^{1}$ with the
  identity component of~$H$.  Then $H_{\alpha}$ is isomorphic to $H$
  via an isomorphism that restricts to $\phi$ on the identity component
  of~$H_{\alpha}$.
\end{proposition}

\begin{proof}
  Let $V_{K}$, $V_{E}$, $V_{F}$ be the subspaces of $V$ defined just
  prior to Proposition~\ref{prop: form gives extension}.  The basis
  elements of $V_{E}$ can be lifted to elements of $H$, which can be
  chosen to be of order~$p$ because $S^{1}$ is a divisible group.  The
  lifts commute since the form is trivial on $V_{E}$. Mapping basis
  elements of $V_{E}$ to their lifts in $H$ gives a monomorphism of
  groups $V_{E}\hookrightarrow H$ whose image we call $W_{E}$.
  Likewise, we can choose lifts $V_{K}\hookrightarrow H$ and
  $V_{F}\hookrightarrow H$, whose images are subgroups $W_{K}$ and
  $W_{F}$ of $H$, respectively.

  Recall that as a set, $H_{\alpha}=S^{1}\times V$, and for $v\in V$,
  we have $v=\vvec{K}+\vvec{E}+\vvec{F}$ as before.  Let $\wvec{K}$,
  $\wvec{E}$, and $\wvec{F}$ be the images of $\vvec{K}$, $\vvec{E}$,
  and $\vvec{F}$ under the lifting homomorphisms of the previous
  paragraph.  We extend the given isomorphism $\phi:S^{1}\rightarrow
  H_{0}$ to a function $\Phi\colon H_{\alpha}\rightarrow H$ by
\begin{equation}  \label{eq: define phi}
\Phi\left(z, \vvec{K}+\vvec{E}+\vvec{F} \right)=
         \phi(z)\,\wvec{K}\wvec{E}\wvec{F}.
\end{equation}
(Note that we write the group operation additively in $V$, which is
abelian, but multiplicatively in $H$, which may not be.)

We assert that $\Phi$ is a group homomorphism. To see that, suppose we
have two elements $(z,v)$ and $(z',v')$ of $H_{\alpha}$. If we multiply
first in $H_{\alpha}$ we get
$\left(\strut zz'\,\alpha(\vvec{F},\vprimevec{E}), \ v+v'\right)$,
and then application of $\Phi$ gives us
\begin{equation}  \label{eq: messy product}
\phi\left(zz'\right)\,\alpha(\vvec{F},\vprimevec{E})
            (\vvec{K}\vprimevec{K})(\vvec{E}\vprimevec{E})(\vvec{F}\vprimevec{F}).
\end{equation}
On the other hand, if we apply $\Phi$ first and then multiply, we get
$\left(\phi(z)\,\vvec{K}\vvec{E}\vvec{F}\right)
           \left(\phi(z')\,\vprimevec{K}\vprimevec{E}\vprimevec{F}\right)$,
which can be rewritten as
\begin{align}   \label{eq: another messy product}
\phi\left(zz'\right)\,\left(\vvec{K}\vprimevec{K}\right)
            \left(\vvec{E}\vvec{F}\vprimevec{E}\vprimevec{F}\right).
\end{align}
To compare \eqref{eq: messy product}
to~\eqref{eq: another messy product}, we need to relate $\vprimevec{E}\vvec{F}$
and $\vvec{F}\vprimevec{E}$.  However, the commutators in $H$ are given exactly
by $\alpha$, so $\vvec{F}\vprimevec{E}=\alpha(\vvec{F},\vprimevec{E})\vprimevec{E}\vvec{F}$, which
allows us to see that
\eqref{eq: messy product} and \eqref{eq: another messy product} are
equal.  We conclude that $\Phi$ is a group homomorphism.

Finally, we need to know that $\Phi$ is a bijection.  To see that
$\Phi$ is surjective, observe that if $h\in H$ maps to $v\in V$ where
$v=\vvec{K}+\vvec{E}+\vvec{F}$, then $h$ and
$\wvec{K}\wvec{E}\wvec{F}$ differ only by some element $z$ of the
central $S^{1}$. Hence every element of $H$ can we written as
$z \wvec{K} \wvec{E} \wvec{F}$ for some $z$, $\wvec{K}$, $\wvec{E}$,
$\wvec{F}$, and $\Phi$ is surjective.  However,
\eqref{eq: define phi}~tells us that $\Phi$ is the isomorphism $\phi$
on the identity component,~$S^{1}$.  Further, $\Phi$ is a surjection
of the finite set of components, hence a bijection of components.  We
conclude that $\Phi$ is an isomorphism.
\\

\end{proof}

\begin{proposition} \label{proposition: forms=groups}
  The commutator form gives a one-to-one correspondence between
  isomorphism classes of abstract projective elementary abelian
  \pdash groups and isomorphism classes of totally isotropic bilinear
  forms over~$\integers/p$.
\end{proposition}

\begin{proof}
  Every totally isotropic form $\alpha$ is realized as the commutator
  form of the group~$H_{\alpha}$.  By
  Propositions~\ref{prop: form gives extension}
  and~\ref{proposition: one-to-one}, if $H$ and $H'$ have isomorphic
  commutator forms $\alpha$ and $\alpha'$, then
\[
H\cong H_{\alpha}\cong H_{\alpha_{'}}\cong H'.
\]
\end{proof}

Proposition~\ref{proposition: forms=groups} allows us to give the
following explicit classification of abstract projective elementary
abelian \pdash groups. This classification result can also be found in
\cite{Griess} Theorem~3.1, though in this section we have given an
elementary and self-contained discussion and proof.

\begin{proposition}  \label{proposition: group classification}
Suppose that $H$ is an abstract projective elementary abelian \pdash group
\[
1\rightarrow S^{1}\rightarrow H\rightarrow V\rightarrow 1.
\]
Let $2s$ be the maximal rank of a symplectic subspace of $V$ under the
commutator form of $H$, and let $t=\rank(V)-2s$. Then $H$ is
isomorphic to $\Gamma_{s}\times\Delta_{t}$.
\end{proposition}

\begin{proof}
  The commutator form of $\Gamma_{s}\times\Delta_{t}$ is isomorphic to
  that of $H$, so the result follows from
  Proposition~\ref{proposition: forms=groups}.
  To interpret the proposition when $s=0$, note that $\Gamma_{0}=S^{1}$,
  so the proposition says that if $s=0$, then
  $H\cong S^{1}\times\Delta_{t}\cong S^{1}\times V$.
\end{proof}

\begin{remark}
  As pointed out by D.~Benson, one could also approach this
  classification result via group cohomology, using the fact that
  equivalence classes of extensions as in
  Definition~\ref{defn: abstract proj elem} correspond to elements of
  $H^2(V; S^1)$.  An argument using the Bockstein homomorphism shows
  that the group $H^2(V; S^1)\cong H^3(V,\Z)$ can be identified with
  the exterior square of $H^1(V,\Z/p)$, which is in turn isomorphic to
  the space of alternating forms $V\times V \to \Z/p$.  The standard
  factor set approach to $H^2$ can be used to identify such a form
  with the commutator pairing of the extension (as in \cite{Brown} or
  \cite{Weibel}). The factor set associated to an extension is
  similarly defined and often identically denoted as the commutator
  pairing, but the two pairings are \emph{not} the same. In
  particular, a factor set need not be bilinear or totally isotropic.
\end{remark}

\section{Initial list of problematic subgroups}
\label{section: problematic subgroups}

Throughout this section, assume that $n=mp^k$ where $m$ and $p$ are coprime.
Suppose that $H$ is a problematic \pdash toral subgroup of $\Un$.
We know from the first part of Theorem~\ref{thm: H elementary abelian}
that $H$ must be a projective elementary abelian \pdash group.
If $H$ contains $S^{1}\subseteq \Un$, we know the possible group isomorphism
classes of $H$ from Proposition~\ref{proposition: group classification}.
The purpose of this section is to use the character criterion of
Theorem~\ref{thm: H elementary abelian} to narrow down the possible
conjugacy classes of $H$ in $\Un$.

In Section~\ref{section: Gamma_k}, we described the projective elementary \pdash subgroup $\Gamma_{k}\subset\Upk$. In this and subsequent sections, we consider the action of $\Gamma_{k}$ on $\Cn$ by a multiple of its ``standard" action on $\Cpk$: the group $\Gamma_{k}$ acts on $\Cn\cong\complexes^{m}\otimes\Cpk$ by acting trivially on $\complexes^{m}$ and by the standard action on~$\Cpk$. In order to streamline notation, we denote this subgroup of $\Un$ by $\Gamma_{k}$ also, since context will indicate the dimension of the ambient space.

Since $\Gamma_{k}\subseteq\Un$ is represented by block diagonal matrices
with blocks $\Gamma_{k}$, we immediately obtain the following from
Lemma~\ref{lemma: Gamma zero character}.

\begin{lemma}  \label{lemma: Diag character}
The character of $\Gamma_{k}\subseteq\Un$ is nonzero only on the elements of
$S^{1}$, where the character is $\chi(s)=ns\in\complexes$.
\end{lemma}

Our goal is to show that if $H$ is a problematic subgroup of $\Un$ where $n=mp^{k}$ and $m$ and $p$ are coprime, then $H$ is a subgroup of $\Gamma_{k}\subset\Un$.
Although $H$ itself may not be finite, we can use
its finite subgroups to get information about $n$ using the following
result from basic representation theory.

\begin{lemma}\label{lemma: n=mp^k}
  Suppose that $G$ is a finite subgroup of $\Un$ and that
  $\character{G}(y)=0$ unless $y=e$. Then $|G|$
  divides $n$, and the action of $G$ on $\Cn$ is by $n/|G|$
  copies of the regular representation.
\end{lemma}

\begin{proof}
The number of copies of an irreducible character $\character{}$ in
$\character{G}$ is given by the inner product
\[
\langle \character{G},\character{}\rangle
    =\frac{1}{|G|} \sum_{y \in G} \character{G}(y) \overline{\character{}(y)}.
\]
Take $\character{}$ to be the character of the one-dimensional trivial
representation of $G$. The only nonzero term in the
summation occurs when $y=e$, and since $\character{G}(e)=n$,
we find $\langle \character{G},\character{}\rangle=n/|G|$.
Since $\langle \character{G},\character{}\rangle$ must be an integer, we find
that $|G|$ divides $n$. To finish, we observe that the
character of $n/|G|$ copies of the regular
representation is the same as $\character{G}$, which finishes the proof.
\end{proof}

We now have all the ingredients we require to prove
the main result of this section.

\begin{theorem} \label{theorem: Non_contractible_implies_subgroup_Gamma_diag}
\subgroupGammadiagtext
\end{theorem}

\begin{proof}
  We know from
  Theorem~\ref{thm: H elementary abelian}\eqref{item: elem abelian}
  that $H$ is an abstract projective elementary elementary abelian
  \pdash group.  By Proposition~\ref{proposition: group classification},
  $H$ is abstractly isomorphic to $\Gamma_{s}\times\Delta_{t}$ for
  some $s$ and $t$, so $H$ contains a subgroup $(\integers/p)^{s+t}$
  (say, $\Acal_{s}\times\Delta_{t}$). By
  Theorem~\ref{thm: H elementary abelian}\eqref{item: character},
  we have the character criterion on~$H$ necessary to apply
  Lemma~\ref{lemma: n=mp^k}, and we conclude that $p^{s+t}$ divides~$n$.
  Since $n=mp^{k}$ with $m$ coprime to~$p$, we necessarily have
  $s+t\leq k$. Hence by Lemma~\ref{lemma: Gamma_k subgroups}, we know
  $\Gamma_{s}\times\Delta_{t}\subseteq\Gamma_{k}$.

  To finish the proof, we compare two representations
  of~$\Gamma_{s}\times\Delta_{t}$. The first is the composite
\[
\Gamma_{s}\times\Delta_{t}\hookrightarrow\Gamma_{k}\subseteq \Un.
\]
This map gives an identification of the identity component of the
abstract group $\Gamma_{s}\times\Delta_{t}$ with the center
$S^{1}\subseteq\Un$, and in terms of this identification, the character
of the representation is $x\mapsto nx$ on the identity component of
$\Gamma_{s}\times\Delta_{t}$ and zero else
(Lemma~\ref{lemma: Diag character}).

To construct the second representation of~$\Gamma_{s}\times\Delta_{t}$,
we construct a map
$\Gamma_{s}\times\Delta_{t}\rightarrow H$. Since $H$ has the same commutator
form as $\Gamma_{s}\times\Delta_{t}$, by
Proposition~\ref{proposition: one-to-one} there is an isomorphism
$\Gamma_{t}\times\Delta_{t}\rightarrow H\subseteq\Un$ that gives the
same map on identity components as
$\Gamma_{s}\times\Delta_{t}\hookrightarrow\Gamma_{k}
     \hookrightarrow\subseteq \Un$.
Hence the character for this representation is also zero off the identity
component (by Theorem~\ref{thm: H elementary abelian}),
and $x\mapsto nx$ on the central~$S^{1}$.

Thus the two representations of
$\Gamma_{s}\times\Delta_{t}$ have the same character, and we conclude
that they are conjugate. Since the image of one is the subgroup~$H$,
and the image of the other is
  $\Gamma_{s}\times\Delta_{t}\subseteq\Gamma_{k}\subseteq\Un$, the
theorem follows.
\end{proof}

\begin{example}
  Suppose that $p$ is an odd prime, and let $n=2p$. Let $H$ be a
  problematic subgroup of $\Uof{2p}$ acting on~$\Lcal_{2p}$. According
  to Theorem~\ref{theorem: Non_contractible_implies_subgroup_Gamma_diag}, the
  subgroup $H$ is conjugate in $\Uof{2p}$ to a subgroup of~$\Gamma_{1}$.  Since in addition we assume that $H$ contains the
  central $S^{1}$, there are only three possibilities for
  $H$: $S^{1}$ itself, $\Gamma_{1}$ acting by two copies of its standard representation, or $S^{1}\times\Delta_{1}\subset\Gamma_{1}$.
\end{example}

\section{Fixed points and joins}
\label{section: joins}

In this section, we begin the work of computing the fixed points of the $p$-toral subgroups of $\Un$ that are identified in
Theorem~\ref{theorem: Non_contractible_implies_subgroup_Gamma_diag}
as potentially problematic.
Throughout this section, let $n=mp^{k}$, and fix an isomorphism $\Cn\cong\complexes^{m}\otimes\Cpk$.
Let $\Gamma_{k}$ act on $\Cn$ by acting trivially on $\complexes^{m}$
and by its standard representation (described in Section~\ref{section: Gamma_k}) on~$\Cpk$. There is also an action of $\Uof{m}$ on $\Cn$ that commutes with the action of~$\Gamma_{k}$, by letting $\Uof{m}$ act by the standard action on $\complexes^{m}$ and trivially on~$\Cpk$. This action passes to an action of $\Uof{m}$ on the fixed point space~$\weakfixed{\Gamma_{k}}$.

Our goal in this section and the next is to establish which subgroups of $\Gamma_{k}\subseteq\Un$ actually have noncontractible fixed points on~$\Lcal_{n}$. A~starting point is provided by the following result of~\cite{Arone-Lesh-Tits} for $\Gamma_{k}\subset\Upk$ acting on~$\Lcal_{p^{k}}$. We will bootstrap this result to fixed points of $\Gamma_{k}\subset\Un$ acting on~$\Lcal_{n}$.

\begin{proposition}[\cite{Arone-Lesh-Tits}]
        \label{proposition: fixed points of Gamma_k}
For $k\geq 1$, the fixed point space of $\Gamma_k$ acting on $\Lcal_{p^k}$ is homotopy equivalent to a wedge of spheres of dimension~$k-1$.
\end{proposition}

Let $X\join Y$ denote the join of the two spaces $X$ and~$Y$. The theorem below establishes a formula for the fixed point space $\weakfixedspecific{\Gamma_{k}}{mp^k}$ that was suggested to us by G.~Arone. Notice that there is no assumption that $m$ and $p$ should be coprime.

\begin{theorem}  \label{theorem: join theorem}
\jointheoremtext
\end{theorem}

We begin by outlining the proof of Theorem~\ref{theorem: join theorem}.
The strategy is to identify a $\Uof{m}$-subcomplex $\tensorsubcomplex$ of $\weakfixedspecific{\Gamma_{k}}{mp^k}$ such that the nerve of $\tensorsubcomplex$ has the $\Uof{m}$-equivariant homotopy type of the join in Theorem~\ref{theorem: join theorem}. Then we establish that $\tensorsubcomplex$ is a $\Uof{m}$-equivariant deformation retraction of~$\weakfixedspecific{\Gamma_{k}}{mp^k}$.

To construct the subcomplex~$\tensorsubcomplex$,
suppose that $\mu$ and $\nu$ are orthogonal decompositions of $\complexes^{m}$
and~$\Cpk$, respectively, with $\class(\mu)=\{v_{1},\dots,v_{s}\}$ and $\class(\nu)=\{w_{1},\dots,w_{t}\}$. We can tensor the components of $\mu$ and $\nu$ to obtain a decomposition of $\complexes^{m}\otimes\Cpk$ that we denote $\mu\otimes\nu$:
\[
\class(\mu\otimes\nu)=\{v_{i}\otimes w_{j}: 1\leq i\leq s \mbox{ and } 1\leq j\leq t\}.
\]
If $\nu$ is weakly fixed by~$\Gamma_{k}$, then so is $\mu\otimes\nu$, and if at least one of $\mu$ and $\nu$ is proper, then $\mu\otimes\nu$ is proper as well.

\begin{definition}
The subposet
$\tensorsubcomplex\subseteq\weakfixedspecific{\Gamma_{k}}{mp^{k}}$ is the set of objects of the form $\mu\otimes\nu$ where $\mu$ is a decomposition of $\complexes^{m}$ and $\nu$ is a weakly $\Gamma_{k}$-fixed decomposition of $\Cpk$, and at least one of $\mu$ and $\nu$ is proper.
\end{definition}

\begin{remark}  \label{remark: path components Z}
It follows from the definition that $\tensorsubcomplex$ is stabilized by the action of $\Uof{m}$ on $\Cn\cong\complexes^{m}\otimes\Cpk$. In fact,
by  \cite{Oliver-p-stubborn} the centralizer of $\Gamma_{k}$ is actually~$\Uof{m}$, and thus
by Corollary~\ref{cor: J fixed path components} the path components of the object space of $\tensorsubcomplex$ are actually $\Uof{m}$-orbits. The same is true of the morphism space of~$\tensorsubcomplex$, and indeed, for the space of $d$-simplices of $\tensorsubcomplex$ for every~$d$.
\end{remark}

To analyze~$\tensorsubcomplex$, we write it as a union of two subposets, each of which is closed under the action of~$\Uof{m}$.
Let $\first$ (resp.~$\second$) denote the subposet
  of~$\tensorsubcomplex$ consisting of the decompositions
  $\lambda=\mu\otimes\nu$ where $\mu$ (resp.~$\nu$) is a proper
  decomposition of~$\complexes^{m}$ (resp.~$\Cpk$). The object space of $\first$ is stabilized by~$\Uof{m}$, and hence (Remark~\ref{remark: path components Z})
  is a union of path components of the object space of~$\tensorsubcomplex$. The same is true of~$\second$, and likewise the morphism spaces of $\first$ and $\second$ are unions of path components of the morphism space of~$\tensorsubcomplex$.

  Any refinement in $\tensorsubcomplex$ of an object in $\first$ is also
  in~$\first$, and likewise any refinement in $\tensorsubcomplex$ of an object in $\second$ is likewise
  in~$\second$. Hence the nerve of $\tensorsubcomplex$ is the union of the nerve
  of $\first$ and the nerve of~$\second$. We construct a
  $\Uof{m}$-equivariant map of diagrams
\begin{equation}  \label{eq: map arrays}
\left(
\begin{array}{ccc}
\Lcal_{m}\times\weakfixedspecific{\Gamma_{k}}{p^{k}}
     & \longrightarrow &\Lcal_{m}\\
\downarrow\\
\weakfixedspecific{\Gamma_{k}}{p^{k}}
\end{array}
\right)\
\xrightarrow{\ \ \ \ \ \ \ }
\ \left(
\begin{array}{ccc}
\first\cap\second & \longrightarrow \first\\
\downarrow\\
\second
\end{array}
\right)
\end{equation}
\medskip
by doing the following.
\begin{itemize}
\item In the upper right corner, we map a proper decomposition
  $\mu$ of~$\complexes^{m}$ to the decomposition $\mu\otimes\Cpk$, which is in~$\first$.

\item In the lower left corner, we map a decomposition $\nu$ in
  $\weakfixedspecific{\Gamma_{k}}{p^{k}}$ to
   $\complexes^{m}\otimes\nu$, which is in~$\second$.

\item In the upper left corner, we map the pair $(\mu,\nu)$
  to $\mu\otimes\nu$, which is in $\first\cap\second$.
\end{itemize}

In~\eqref{eq: map arrays}, we would like to relate the homotopy pushout of the left diagram to the strict pushout of the right diagram
 (which is $\first\cup\second$). First we establish commutativity of the map of diagrams in~\eqref{eq: map arrays}. We show that at each corner, \eqref{eq: map arrays}~is a homotopy equivalence, and then we show that this statement remains true when we take fixed points under the action of a subgroup $H\subseteq\Uof{m}$. Finally, we show that in the right-hand diagram of~\eqref{eq: map arrays}, the map from the homotopy pushout to the strict pushout is a homotopy equivalence and remains so after taking fixed points under $H\subseteq\Uof{m}$.

\begin{lemma} \label{lemma: commuting ladder}
The following ladder is $\Uof{m}$-equivariantly homotopy
commutative, and the vertical arrows are equivalences of $\Uof{m}$-spaces:
\begin{equation} \label{diag: lemma ladder}
\begin{CD}
     \weakfixedspecific{\Gamma_{k}}{p^{k}}
     @<<<\Lcal_{m}\times\weakfixedspecific{\Gamma_{k}}{p^{k}}
     @>>> \Lcal_{m}\\
@V{\simeq}VV   @V{\cong}VV  @V{\simeq}VV\\
\second@<<< \first\cap\second @>>> \first .
\end{CD}
\end{equation}
\end{lemma}

\begin{proof}
To establish homotopy commutativity, consider
a pair $(\mu,\nu)$ in $\Lcal_{m}\times\weakfixedspecific{\Gamma_{k}}{p^{k}}$.
Going clockwise around the right-hand square of
\eqref{diag: lemma ladder} yields the decomposition $\mu\otimes\Cpk$
in~$\first$.
Going around that square counterclockwise yields the decomposition
$\mu\otimes\nu$ in~$\first$. There is a natural coarsening
$\mu\otimes\nu\rightarrow\mu\otimes\Cpk$, and the coarsening morphism is stabilized by~$\Uof{m}$. Hence the right-hand square of
diagram~\eqref{diag: lemma ladder} induces
a $\Uof{m}$-homotopy commutative diagram of nerves. Similarly,
following $(\mu,\nu)$ clockwise around
the left-hand square gives $\mu\otimes\nu$ in~$\second$, and going counterclockwise gives $\complexes^{m}\otimes\nu$ in~$\second$, and there is a natural $\Uof{m}$-equivariant homotopy given by the coarsening $\mu\otimes\nu\rightarrow\complexes^{m}\otimes\nu$.

The vertical equivalences result from similar arguments, as follows.
\begin{enumerate}
\item To see that the right-hand vertical map,
    $\Lcal_{m}\rightarrow\first$, is a homotopy equivalence of
$\Uof{m}$-spaces, consider that   $\mu\otimes\nu\mapsto\mu\otimes\Cpk$
gives a natural retraction. The coarsening map
$\mu\otimes\nu\rightarrow\mu\otimes\Cpk$ gives a $\Uof{m}$-equivariant homotopy between the retraction and the identity.

\item Likewise, we consider the left-hand map, $\weakfixedspecific{\Gamma_{k}}{p^{k}}\rightarrow\second$. The map
$\mu\otimes\nu\mapsto\complexes^{m}\otimes\nu$ gives a natural retraction, and the coarsening map $\mu\otimes\nu\rightarrow\complexes^{m}\otimes\nu$ is a $\Uof{m}$-equivariant homotopy making it a deformation retraction.

\item The map $\Lcal_{m}\times\weakfixedspecific{\Gamma_{k}}{p^{k}}
      \longrightarrow \first\cap\second$ is a $\Uof{m}$-equivariant isomorphism of posets.
\end{enumerate}
\end{proof}

As a consequence of Lemma~\ref{lemma: commuting ladder}, we know that the map
of diagrams in \eqref{eq: map arrays} induces a map between the homotopy pushouts that is $\Uof{m}$-equivariant and a homotopy equivalence. However,
to get equivalences of $\Uof{m}$-spaces, we need to know what happens after taking fixed points of a subgroup $J\subseteq\Uof{m}$. Hence we need to know the relationship between taking fixed points and taking homotopy pushouts. We give an argument that we learned from C. Malkiewich~\cite{CM-notes}.

\begin{proposition}     \label{proposition: fixed points of mapping cylinder}
Let $f:A\longrightarrow B$ be an equivariant map of spaces with an action of a group~$J$, and let
$\Cylinder(f)$ denote the mapping cylinder of~$f$. Let $C=\Cylinder\left(f^{J}\right)$ be the mapping cylinder
of the function $A^{J}\longrightarrow B^{J}$ given by restricting $f$ to
$J$-fixed points,
 and let $D=\left(\Cylinder(f)\right)^{J}$ be the fixed point space of the action of $J$ on $\Cylinder(f)$.
Then the natural map $C\longrightarrow D$ is a homeomorphism.
\end{proposition}

\begin{proof}
The space $C$ is the quotient space of
$\left(A^{J}\times [0,1]\right)\coprod B^{J}$
by the relation $(a,1)\simeq f(a)$. The inclusions
$A^{J}\times [0,1]\hookrightarrow A\times [0,1]$ and $B^{J}\hookrightarrow B$
induce a natural map $C\rightarrow \Cylinder(f)$, whose image is contained in~$D$, and which is continuous by definition of the quotient topology
on~$C$. Further, using the fact that
$A^{J}$ and $B^{J}$ are closed in $A$ and $B$, respectively,
we can check that $C\rightarrow D$ is a closed map.
A routine check verifies that $C\rightarrow D$ is a bijection. We conclude that $C\rightarrow D$ is a homeomorphism.
\end{proof}

\begin{corollary}     \label{corollary: fixed points and hocolim}
Suppose that
\[
\begin{CD}
A@>>> B\\
@VVV @VVV\\
C @>>> D
\end{CD}
\]
is a homotopy pushout diagram of spaces with an action of a group~$J$, and $J$-equivariant maps. Then
\begin{equation}
\begin{CD}
A^J@>>> B^J\\
@VVV @VVV\\
C^J@>>>D^{J}
\end{CD}
\end{equation}
is also a homotopy pushout diagram.
\end{corollary}

\begin{proof}
The corollary is a consequence of Proposition~\ref{proposition: fixed points of mapping cylinder} once we have replaced $B$ and $C$ with mapping cylinders and $D$ with the double mapping cylinder.
\end{proof}

\begin{proposition} \label{proposition: subcomplex is a join}
The nerve of $\tensorsubcomplex$ is $\Uof{m}$-equivariantly homotopy equivalent to $\Lcal_{m}\join\weakfixedspecific{\Gamma_{k}}{p^{k}}$.
\end{proposition}

\begin{proof}
Lemma~\ref{lemma: commuting ladder}
and Corollary~\ref{corollary: fixed points and hocolim} tell us that
the map of diagrams in \eqref{eq: map arrays} induces an equivalence of $\Uof{m}$-spaces on homotopy pushouts, i.e., a $\Uof{m}$-equivariant map that is a homotopy equivalence on the fixed point space of any subgroup $J\subseteq\Uof{m}$. The homotopy pushout of the left diagram in \eqref{eq: map arrays} is $\Lcal_{m}\join\weakfixedspecific{\Gamma_{k}}{p^{k}}$, and we need to relate this space to the strict pushout (not the homotopy pushout) of the right diagram in \eqref{eq: map arrays}, the strict pushout being $\first\cup\second=\tensorsubcomplex$. Hence the proposition follows once we show that the natural map from the homotopy pushout to the strict pushout for the diagram
\[
\begin{CD}
\first\cap\second @>>> \first\\
@VVV\\
\second
\end{CD}
\]
is a homotopy equivalence, and remains so after taking fixed points for any subgroup $J\subseteq\Uof{m}$. This statement follows if we can prove that for all subgroups $J$ of~$\Uof{m}$, the space $\first^{J}$ is Reedy cofibrant and $(\first\cap\second)^{J}\rightarrow\second^{J}$ is a Reedy cofibration.

To establish that $\first$ itself is Reedy cofibrant, consider the space $\first_{d}$ of $d$-simplices of~$\first$. Let $L_{d}\left(\first\right)$ denote the space of degenerate $d$-simplices of~$\first$ (i.e., the $d$th latching object for~$\first$).
We must show that
\begin{equation}  \label{eq: inclusion of latching}
L_{d}\first\longrightarrow \first_{d}
\end{equation}
is a cofibration of topological spaces. We assert that $L_{d}\first$ is, in fact, a union of path components of~$\first_{d}$.
The complex $\first\subseteq\tensorsubcomplex$ is stabilized by~$\Uof{m}$, and by Remark~\ref{remark: path components Z} the path components of the $d$-simplices of $\tensorsubcomplex$ are $\Uof{m}$-orbits. Hence
$\first_{d}$~is a disjoint union of $\Uof{m}$-orbits, and the same is true
of~$\first_{d-1}$. The degeneracy maps of $\first$ are $\Uof{m}$-equivariant, and it follows that their images are unions of path components. The space $L_{d}\first$ is the union of such images, and is therefore also a union of path components of~$\first_d$, establishing that $\first$ is Reedy cofibrant.

We need to know that the statements of the previous paragraph remain true when we replace $\first$ by $\first^{J}$, where $J$ is any subgroup of~$\Uof{m}$.
By Corollary~\ref{cor: J fixed path components}, the fixed point space of the action of $J$ on a $\Uof{m}$-orbit has path components that are orbits under the action of~$\CentIdent{J}$, the identity component of the centralizer of $J$ in~$\Uof{m}$. We can now follow exactly the same reasoning as in the previous paragraph to conclude that $\first^{J}$ is Reedy cofibrant.

We claim that the map $\first\cap\second\rightarrow \second$ is a Reedy cofibration for essentially the same reasons. We must show that
for each~$d$, the map from the pushout of
the diagram
\begin{equation} \label{diagram: latching objects}
\begin{CD}
\degenerate_d (\first\cap\second) @>>> (\first\cap\second)_d\\
@VVV\\
\degenerate_d (\second)
\end{CD}
\end{equation}
 to $\second_{d}$ is a cofibration in topological spaces. But $\first_{d}$, $\second_{d}$, and $(\first\cap\second)_{d}$ are disjoint unions of $\Uof{m}$-orbits, as are their subspaces of degenerate simplices. As a consequence, each of the spaces in \eqref{diagram: latching objects} is a union of path components of~$\second_d$, and the maps are inclusions. So the pushout is likewise a union of path components of~$\second_d$, and its inclusion into $\second_d$ is a cofibration.

We also need to know that the map remains a Reedy cofibration after taking
$J$-fixed points for any subgroup $J$ of~$\Uof{m}$, and the argument is obtained by combining the previous two paragraphs, since taking $J$-fixed point spaces results in spaces of $d$-simplices that are orbits of~$\CentIdent{J}$.
\end{proof}

Now that we have constructed the $\Uof{m}$-equivariant subcomplex $\tensorsubcomplex$ of~$\weakfixedspecific{\Gamma_{k}}{mp^k}$ with the desired homotopy type, we need to prove that there is a deformation retraction of $\Uof{m}$-spaces from $\weakfixedspecific{\Gamma_{k}}{mp^k}$ to~$\tensorsubcomplex$. We obtain it in steps, by constructing an interpolating subcategory between $\tensorsubcomplex$ and~$\weakfixedspecific{\Gamma_{k}}{mp^{k}}$ using the following definition.

\begin{definition}    \label{definition: uniform isotropy}
Suppose that $H\subseteq\Un$ and $\lambda$ is an object of $\weakfixed{H}$.
\begin{enumerate}
\item For $v\in\class(\lambda)$, define the
\defining{$H$-isotropy group} of~$v$ as $\{g\in H: gv=v\}$.
\item
 We say that $\lambda$ has \defining{uniform $H$-isotropy} if every $v\in\class(\lambda)$ has the same $H$-isotropy group. In this case, we write $\isogroupof{\lambda}$ for the $H$-isotropy group of each element of~$\class(\lambda)$.
\item We say that $\lambda$ is \defining{$H$-isotypical} if $\lambda$ has uniform $H$-isotropy and $\isogroupof{\lambda}$ acts isotypically on each component $v\in\class(\lambda)$.
\end{enumerate}
\end{definition}

There is an easy criterion guaranteeing uniform isotropy.

\begin{lemma}   \label{lemma: transitive means uniform}
Suppose $\lambda\in\Obj\weakfixed{H}$ has the property that for some $v\in\class(\lambda)$, the $H$-isotropy subgroup of $v$ is normal in~$H$. If $H$ acts  transitively on $\class(\lambda)$, then
$\lambda$ has uniform $H$-isotropy.
\end{lemma}

\begin{proof}
Since the action of $H$ on $\class(\lambda)$ is transitive, the $H$-isotropy subgroups of elements of $\class(\lambda)$ are all conjugate in $H$ to the isotropy group of~$v$, which is assumed to be normal. We conclude that they are all actually the same.
\end{proof}

\begin{corollary}
All objects in $\weakfixedspecific{\Gamma_{k}}{p^k}$ have uniform isotropy, as do objects in~$\tensorsubcomplex$.
\end{corollary}

\begin{proof}
Since $\Gamma_{k}$ acts irreducibly on $\complexes^{k}$, it necessarily acts transitively on the components of any weakly  $\Gamma_{k}$-fixed decomposition of~$\Cpk$. An object $\mu\otimes\nu$ in $\tensorsubcomplex$ has the same $\Gamma_{k}$-isotropy group as~$\nu$.
\end{proof}

Since objects in $\tensorsubcomplex$ have uniform $\Gamma_{k}$-isotropy, whereas objects in $\weakfixedspecific{\Gamma_{k}}{mp^{k}}$ may not, we consider an interpolating subcomplex that focuses on uniform isotropy.

\begin{definition}
Let $\Uniform{\Gamma_{k}}{mp^k}$ be the subposet of $\weakfixedspecific{\Gamma_{k}}{mp^k}$
given by objects with uniform $\Gamma_{k}$-isotropy.
\end{definition}

\begin{proposition}   \label{proposition: inclusion of uniform subposet}
The inclusion $\Uniform{\Gamma_{k}}{mp^k}\subseteq\weakfixedspecific{\Gamma_{k}}{mp^k}$
induces a homotopy equivalence of $\Uof{m}$-spaces on nerves.
\end{proposition}

\begin{proof}
  As usual, we define a $\Uof{m}$-equivariant deformation retraction.
  Let $\lambda$ be a
  decomposition that is weakly fixed by $\Gamma_{k}$, say
  $\class(\lambda)=\{v_{1},\dots,v_{j}\}$.
  Let $\isogroupof{v_1},\dots,\isogroupof{v_j}$ be
  the $\Gamma_{k}$-isotropy subgroups of $v_{1},\dots,v_{j}$,
  respectively. Each of them contains $S^{1}$ and is therefore
  normal in~$\Gamma_{k}$.
  Let $\isogroupof{\lambda}\subseteq\Gamma_{k}$ be the product,
  $\isogroupof{\lambda}=\isogroupof{v_1}\dots \isogroupof{v_j}$.
  The construction of $\isogroupof{\lambda}$ is $\Uof{m}$-invariant, since $\Uof{m}$ centralizes~$\Gamma_{k}$.

  Recall that $\glom{\lambda}{\isogroupof{\lambda}}$ denotes the strongly
$\isogroupof{\lambda}$-fixed coarsening of $\lambda$ created by summing components in the same orbit of the action of $\isogroupof{\lambda}$ on~$\class(\lambda)$ (see
Definition~\ref{orbit_partition_def}). Consider the assignment
\[
\lambda\mapsto\glom{\lambda}{\isogroupof{\lambda}},
\]
and observe that it is~$\Uof{m}$-equivariant.

First we check that this assignment actually lands
in~$\weakfixedspecific{\Gamma_{k}}{mp^k}$. Because $S^{1}\subseteq
\isogroupof{\lambda}$, we know
$\isogroupof{\lambda}\triangleleft\Gamma_{k}$.
Lemma~\ref{lemma: glom} then tells us that
$\glom{\lambda}{\isogroupof{\lambda}}$ is weakly fixed by~$\Gamma_{k}$.
We also need to check that $\glom{\lambda}{\isogroupof{\lambda}}$
is in fact proper. If not, then $\isogroupof{\lambda}$ acts
transitively on $\class(\lambda)$, and therefore $\Gamma_{k}$ does
also. By Lemma~\ref{lemma: transitive means uniform}, we have
$\isogroupof{v_1}=\dots=\isogroupof{v_j}$, so
$\isogroupof{\lambda}=\isogroupof{v_1}$.
But then $\isogroupof{v_1}$ acts transitively on $\class(\lambda)$, which
is a contradiction since $\isogroupof{v_1}$ fixes $v_{1}$
and $\lambda$ is proper.

We now check that the assignment
$\lambda\mapsto\glom{\lambda}{\isogroupof{\lambda}}$ is continuous,
by considering the $\Gamma_{k}$-isotropy of decompositions in the same path component
of $\Obj(\Lcal_{mp^{k}})^{\Gamma_{k}}$ as~$\lambda$.
By Corollary~\ref{cor: J fixed path components}, the path component of $\lambda$
consists of elements $c\lambda$ where $c$ is in the identity component of
$\Cen_{\Uof{mp^k}}\left(\Gamma_{k}\right)$. However, if $v\in\class(\lambda)$ then
$\isogroupof{c\,v}=c\left(\isogroupof{v}\right)c^{-1}=\isogroupof{v}$. As a result,
$\isogroupof{\lambda}=\isogroupof{c\lambda}$. Therefore the same subgroup is being
used to coarsen the entire path component of~$\lambda$, and we already know that
this operation is continuous from Lemma~\ref{lemma: continuity of glom}.

Next we verify that the assignment
$\lambda\mapsto\glom{\lambda}{\isogroupof{\lambda}}$
respects coarsenings: we must check that if $\mu\leq\lambda$, then
$\glom{\mu}{\isogroupof{\mu}}\leq\glom{\lambda}{\isogroupof{\lambda}}$.
Suppose that $w\in\class(\mu)$ and that $w\subseteq v\in\class(\lambda)$, and further
suppose that $\gamma\in\isogroupof{w}$. Since $\isogroupof{w} \subset \Gamma_k$ and $\Gamma_k$ weakly fixes $\lambda$, we have that $\gamma v$
is either equal to $v$ or orthogonal to $v$, whence $\gamma v=v$.
Hence $\gamma\in\isogroupof{v}$, establishing that
$w\subseteq v$ implies $\isogroupof{w}\subseteq\isogroupof{v}$.
As a result, $\isogroupof{\mu}\subseteq\isogroupof{\lambda}$. We
conclude that
$\glom{\mu}{\isogroupof{\mu}}\leq\glom{\lambda}{\isogroupof{\mu}}
\leq\glom{\lambda}{\isogroupof{\lambda}}$.

The coarsening $\lambda\rightarrow\glom{\lambda}{\isogroupof{\lambda}}$
gives a $\Uof{m}$-equivariant homotopy
from the identity functor on $\weakfixedspecific{\Gamma_{k}}{mp^k}$ to the composition of retraction and inclusion, and the lemma follows.
\end{proof}

At this point, we have
\[
\tensorsubcomplex\subseteq\Uniform{\Gamma_{k}}{mp^k}\subseteq\weakfixedspecific{\Gamma_{k}}{mp^k}
\]
and the second inclusion is an equivalence of $\Uof{m}$-spaces. Next, we interpolate again by defining $\isotypicsubcomplex$ as the subcomplex of $\Uniform{\Gamma_{k}}{mp^k}$ of decompositions $\lambda$ that are $\isogroupof{\lambda}$-isotypical, where $\isogroupof{\lambda}$ is the (uniform) $\Gamma_k$-isotropy of $\lambda$.

\begin{proposition}     \label{proposition: isotypic to unif}
The inclusion $\isotypicsubcomplex\hookrightarrow\Uniform{\Gamma_{k}}{mp^k}$
induces a homotopy equivalence of $\Uof{m}$-spaces on nerves.
\end{proposition}

\begin{proof}
  Again we define a $\Uof{m}$-equivariant deformation retraction. Let $\lambda$ be a
  decomposition in~$\Uniform{\Gamma_{k}}{mp^k}$ with
  $\Gamma_{k}$-isotropy $\isogroupof{\lambda}\subseteq\Gamma_{k}$. Now
  consider the assignment
  $\lambda\mapsto\isorefine{\lambda}{\isogroupof{\lambda}}$, which we assert
is continuous. As in the proof of
Proposition~\ref{proposition: inclusion of uniform subposet}, continuity follows from the
fact that $\isogroupof{\lambda}$ is constant on each path component, and
isotypical refinement with respect to a subgroup is continuous on each path
component (Lemma~\ref{lemma: continuity of isorefine}). Further, the value of $\isogroupof{\lambda}$ does not change when we act on $\lambda$ by an element of~$\Uof{m}$, because $\Uof{m}$ centralizes~$\Gamma_{k}$. And because $\Uof{m}$ also necessarily centralizes $\isogroupof{\lambda}$,
the assignment $\lambda\mapsto\isorefine{\lambda}{\isogroupof{\lambda}}$ is also $\Uof{m}$-equivariant.

  To check that $\lambda\mapsto\isorefine{\lambda}{\isogroupof{\lambda}}$
  is natural in $\lambda$,
  suppose given $\mu\leq\lambda$ with uniform $\Gamma_{k}$-isotropy
  $\isogroupof{\mu}$ and $\isogroupof{\lambda}$, respectively.  As in the proof of Proposition~\ref{proposition: inclusion of uniform subposet}, $\isogroupof{\mu}\subseteq\isogroupof{\lambda}$. We
  need to check that $\isorefine{\mu}{\isogroupof{\mu}}$ is a
  refinement of~$\isorefine{\lambda}{\isogroupof{\lambda}}$. Suppose
  we have $w\in\class(\mu)$ with $w\subseteq v\in\class(\lambda)$.  We
  need to prove that for every $\isogroupof{\mu}$-isotypical summand
  of $w$, there exists a $\isogroupof{\lambda}$-isotypical summand
  of~$v$ that contains it.  It is sufficient to show that
 non-isomorphic irreducible representations of $\isogroupof{\lambda}$
  contained in $v$ cannot contain isomorphic irreducible
  representations of~$\isogroupof{\mu}$.

  However, any $\isogroupof{\lambda}$-irreducible subspace of $v$ is
  contained in the restriction of the standard representation of
  $\Gamma_{k}$ to~$\isogroupof{\lambda}$.  By Frobenius reciprocity
  (\cite[Theorem 9.9]{Knapp}), the restriction of the standard representation of
  $\Gamma_{k}$ to $\isogroupof{\mu}$ splits as the sum
  of pairwise nonisomorphic $\isogroupof{\mu}$-irreducibles. Since
  $\isogroupof{\mu}\subseteq\isogroupof{\lambda}$, we conclude that
  nonisomorphic representations of $\isogroupof{\lambda}$ contained
  in~$v$ cannot contain isomorphic representations of~$\isogroupof{\mu}$.

Finally, the coarsening morphism
$\isorefine{\lambda}{\isogroupof{\lambda}}\rightarrow\lambda$ provides the
necessary $\Uof{m}$-equivariant natural transformation from the composition of retraction and inclusion to the identity functor.
\end{proof}

We continue with a result on decompositions of tensor products. Since the particular properties of $\Gamma_{k}$ are not needed, we use a more general statement.
\begin{lemma}   \label{lemma: well-defined tensor}
  Suppose that $H\subseteq\Uof{i}$ acts irreducibly on~$\complexes^{i}$,
  and let~$\complexes^{m}$ have trivial $H$-action.
  If $v$ is an $H$-invariant subspace of
  $\complexes^{m}\otimes\complexes^{i}$, then there exists a
  well-defined subspace $w_{v}\subseteq\complexes^{m}$ such that
  $v=w_{v}\otimes\complexes^{i}$. The assignment $v\mapsto w_{v}$ is
  natural in~$v$, is continuous, preserves orthogonality, and is $\Uof{m}$-equivariant.
\end{lemma}

\begin{corollary}
If $n=mp^{k}$ then there is a $\Uof{m}$-equivariant isomorphism
$\strongfixedspecific{\Gamma_{k}}{mp^{k}}\cong\Lcal_{m}
$.
\end{corollary}

\begin{proof}[Proof of Lemma~\ref{lemma: well-defined tensor}]
By Schur's Lemma, there is an isomorphism
\[
\End_{\complexes}\left(\complexes^{m}\right)
\xrightarrow{\ \cong\ }
\End_{H}\left(\complexes^{m}\otimes\complexes^{i}\right)
\]
given by the function $f\mapsto f\otimes\complexes^{i}$, and the function is $\Uof{m}$-equivariant by inspection.

The idempotent
$e_{v}\in\End\left(\complexes^{m}\otimes\complexes^{i}\right)$
corresponding to orthogonal projection to~$v$ is $H$-equivariant.
Hence $e_{v}$ has the form $f_{v}\otimes\complexes^{i}$ for some
$f_{v}\in\End_{\complexes}\left(\complexes^{m}\right)$.
Since $f_{v}$ is necessarily an idempotent, it defines a subspace
$w_{v}=\im\left(f_{v}\right)$, and $v=w_{v}\otimes\complexes^{i}$.
The three assignments $v\mapsto e_{v}\mapsto f_{v}\mapsto w_{v}$ are
each continuous. Finally, $v_{1}\perp v_2$
implies that $w_{v_{1}}\perp w_{v_{2}}$ (for example, because the
corresponding idempotents compose to zero).
\end{proof}

We now have everything we need to prove
Theorem~\ref{theorem: join theorem}.

\begin{proof}[Proof of Theorem~\ref{theorem: join theorem}]
We have a sequence of poset inclusions,
\[
\tensorsubcomplex\subseteq\isotypicsubcomplex\subseteq
     \Uniform{\Gamma_{k}}{mp^k}\subseteq\weakfixedspecific{\Gamma_{k}}{mp^k}
\]
and we have already established that the second two inclusions
induce homotopy equivalences of $\Uof{m}$-spaces on nerves (Proposition~\ref{proposition: isotypic to unif}
and Proposition~\ref{proposition: inclusion of uniform subposet}, respectively).
To finish the proof, we show that $\tensorsubcomplex\subseteq\isotypicsubcomplex$ is a $\Uof{m}$-equivariant isomorphism.

  Suppose that $\lambda$ is a decomposition of
  $\Cn\cong\complexes^{m}\otimes\Cpk$ that lies in~$\isotypicsubcomplex$, that is, $\lambda$ is
  weakly fixed by~$\Gamma_{k}$, has uniform $\Gamma_{k}$-isotropy
  group~$\isogroupof{\lambda}$, and is
  $\isogroupof{\lambda}$-isotypic. We will prove that there exists a decomposition
  $\mu$ of $\complexes^{m}$ and a weakly $\Gamma_{k}$-fixed
  decomposition $\nu$ of $\Cpk$ such that
  $\lambda=\mu\otimes\nu$. Note that if $\lambda$ is proper, then one of
$\mu$ or $\nu$ (but not both) could have a single component.

  As in the proof of Proposition~\ref{proposition: inclusion of uniform subposet},
  Frobenius reciprocity tells us that the restriction of the
  standard representation of $\Gamma_{k}$ on $\Cpk$ to
  the subgroup $\isogroupof{\lambda}$ splits as a sum of
   pairwise nonisomorphic
  $\isogroupof{\lambda}$-irreducible subspaces, say $v_1, v_2,\ldots, v_r$.
  These subspaces are orthogonal, and we use them to define the
  decomposition $\nu$ of $\Cpk$
  with $\class(\nu)=\{v_1, v_2,\ldots, v_r\}$.
  Note that $\nu$ is weakly fixed by $\Gamma_k$ since
  $\isogroupof{\lambda}\triangleleft\Gamma_{k}$.
  The elements of $\class(\nu)$ are permuted transitively
  by the action of~$\Gamma_{k}$ because $\Gamma_{k}$ acts irreducibly
  on~$\Cpk$.

  Recall that by assumption, the components of $\lambda$ are isotypical representations of $\isogroupof{\lambda}$. Fix $v\in\class(\nu)$, and consider the components of $\lambda$ that are $\isogroupof{\lambda}$-isotypical for the irreducible representation~$v$,
  say $c_{1},\dots,c_{s}\in\class(\lambda)$. Each one is an
  $\isogroupof{\lambda}$-subrepresentation of
  of~$\complexes^m \otimes v$, which is the canonical $v$-isotypical summand of $\complexes^{m}\otimes\Cpk$. Thus by
  Lemma~\ref{lemma: well-defined tensor}, there exist
  subspaces $w_{1},\dots,w_{s}$ of $\complexes^m$ such that
  $c_{1} = w_{1} \otimes v$,\dots,$c_{s} = w_{s} \otimes v$.
  We take $\mu$ to be the decomposition defined by
  $\class(\mu)=\{w_{1},\dots,w_{s}\}$.

  Because every component of $\lambda$ is isotypical for a unique element
  of~$\class(\nu)$, and $\Gamma_{k}$ acts transitively
  on~$\class(\nu)$, we know that for every $c\in\class(\lambda)$ there
  exists $\gamma\in\Gamma_{k}$ such that
  $\gamma c\subseteq \complexes^{m}\otimes v$. Since
  $\gamma c\in\class(\lambda)$, we must have that
  $\gamma c$ is equal to some $w_{i}\otimes v$.

  Hence the set $\class(\lambda)$ must be the union of the orbits of
  $w_{1} \otimes v$,\dots,$w_{s} \otimes v$ under the action
  of $\Gamma_{k}$. These orbits are, respectively,
  $w_{1}\otimes\nu$,\dots,$w_{s}\otimes\nu$. We conclude that
  $\lambda=\mu\otimes\nu$.
\end{proof}

\section{Proof of the classification theorem}
\label{section: proof of classification theorem}

In this section, we prove the classification theorem for problematic \pdash toral subgroups of~$\Un$, Theorem~\ref{theorem: new classification theorem}.
Recall that if $m$ and $p$ are coprime, then any problematic $p$-toral subgroup of $\Uof{mp^k}$ is subconjugate to $\Gamma_{k}$ acting on $\complexes^{mp^k}$ by $m$
copies of the standard representation of $\Gamma_{k}\subseteq\Upk$
(Theorem~\ref{theorem: Non_contractible_implies_subgroup_Gamma_diag}).
Furthermore, all subgroups of $\Gamma_{k}$ that contain $S^{1}$ have the form $\Gamma_{s}\times\Delta_{t}$, where $s+t\leq k$
(Proposition~\ref{proposition: group classification}).
Hence, we can use Theorem~\ref{theorem: join theorem} to obtain a reduction of the classification theorem as follows.

\begin{proposition} \label{proposition: reduction to Delta_t}
Let $H$ be a subgroup of $\Gamma_{k}\subseteq\Un$ acting by a multiple of the standard representation of $\Gamma_{k}\subseteq\Upk$. Suppose that $H\cong\Gamma_{s}\times\Delta_{t}$, and let $r=n/p^{s+t}$. Then $\weakfixed{H}$ is contractible (respectively, mod $p$ acyclic) if and only if the $s$-fold suspension of $\weakfixedspecific{\Delta_{t}}{rp^t}$ is contractible (respectively, mod $p$ acyclic).
\end{proposition}

\begin{proof}
If $s=0$, then we are considering $H\cong S^{1}\times\Delta_{t}$ and $n=rp^{t}$.
Hence $\weakfixed{H}
        =\weakfixedspecific{\Delta_{t}}{rp^t}$,
so the proposition is tautologically true.

For $s>0$, Theorem~\ref{theorem: join theorem} gives an equivalence of
$\Uof{rp^t}$-spaces
\[
\weakfixed{\Gamma_s}\simeq\Lcal_{rp^t}\join\weakfixedspecific{\Gamma_s}{p^s}.
\]
Taking fixed points under the action of $\Delta_{t}\subseteq\Uof{rp^t}$ gives
\[
\weakfixed{H}
\simeq\weakfixedspecific{\Delta_{t}}{rp^t}\join\weakfixedspecific{\Gamma_s}{p^s}.
\]
Recall that $\weakfixedspecific{\Gamma_s}{p^s}$ is a wedge of spheres of dimension $s-1$ (Proposition~\ref{proposition: fixed points of Gamma_k}).
Choosing basepoints gives
$X\join Y\simeq\Sigma(X\wedge Y)$ for any $X$ and~$Y$, so $\weakfixed{H}$ has the homotopy type of a wedge of $s$-fold suspensions of~$\weakfixedspecific{\Delta_{t}}{rp^t}$.
\end{proof}

Most of the remainder of this section is devoted to studying fixed points of $\Delta_{t}$ acting on~$\Lcal_{rp^t}$, in preparation for assembling the classification theorem at the end of the section. When $r=1$, we can quote the following result.

\begin{proposition}[\cite{Arone-Lesh-Tits}]
        \label{proposition: fixed points of Delta_t}
For $t\geq 1$, the fixed point space of $\Delta_t$ acting on $\Lcal_{p^t}$ contains, as a retract, a wedge of spheres of dimension~$(t-1)$.
\end{proposition}

Next we look at the special case of $\weakfixedspecific{\Delta_{t}}{mp^t}$ where
$m$ is a positive power of a prime $q$ different from~$p$.
In this case, we do not get contractibility, but we do have mod~$p$ acyclicity. Because we have two primes in play, we specify them explicitly in the notation for the following proposition. We write $\Delta_{t}(p)\cong(\integers/p)^t$, and we write $\Gamma_{r}(q)$ for the mod $q$ irreducible projective elementary abelian $q$-group of $\Uof{q^r}$,
\[
1\rightarrow S^{1}\rightarrow\Gamma_{r}(q)\rightarrow(\integers/q)^{2r}\rightarrow 1.
\]

 \begin{proposition}     \label{prop: full Delta is problematic}
 If $m=q^r$ where $q$ is a prime different from $p$ and $r>0$, then $\weakfixedspecific{\Delta_t(p)}{mp^t}$ is mod $p$ acyclic, but has nontrivial mod $q$ homology, and in particular is not contractible.
 \end{proposition}

\begin{proof}
Mod $p$ acyclicity of $\weakfixedspecific{\Delta_t(p)}{mp^t}$ follows from Smith theory, since $\Lcal_{mp^t}$ is a finite complex and is mod $p$ acyclic by
Corollary~\ref{corollary: contractible}.
To prove the statement about mod $q$ homology of~$\weakfixedspecific{\Delta_t(p)}{mp^t}$, we use Smith theory again, this time
in reverse and for mod~$q$ homology, as follows.

We reverse the roles of $p$ and~$q$, and we consider the action of $\Delta_{t}(p)\times\Gamma_{r}(q)$ on $\complexes^{p^t}\otimes\complexes^{q^r}$.
By  Theorem~\ref{theorem: join theorem} applied to $\Gamma_{r}(q)$, with $\Delta_{t}(p)\subseteq \Uof{p^t}$, we find that
\[
\weakfixedspecific{\Delta_{t}(p)\times\Gamma_{r}(q)}{p^{t}q^{r}}
\simeq
\weakfixedspecific{\Delta_{t}(p)}{p^t}\join\weakfixedspecific{\Gamma_{r}(q)}{q^r}.
\]
By Proposition~\ref{proposition: fixed points of Gamma_k},
$\weakfixedspecific{\Gamma_{r}(q)}{q^r}$ is a wedge of spheres.
On the other hand, by Proposition~\ref{proposition: fixed points of Delta_t},
$\weakfixedspecific{\Delta_{t}(p)}{p^t}$ has a wedge of spheres as a retract. Hence
$\weakfixedspecific{\Delta_{t}(p)\times\Gamma_{r}(q)}{p^{t}q^{r}}$ also has a wedge of spheres as a retract, and therefore has nonzero mod~$q$ homology.

But in fact,
\[
\weakfixedspecific{\Gamma_{r}(q)\times\Delta_{t}(p)}{q^{r}p^{t}}
=
\left(
\weakfixedspecific{\Delta_{t}(p)}{q^{r}p^{t}}
\right)^{\Gamma_{r}(q)/S^1}.
\]
If $\weakfixedspecific{\Delta_{t}(p)}{q^{r}p^{t}}$ were mod~$q$ acyclic, then
we could apply Smith theory to the finite complex $\weakfixedspecific{\Delta_{t}(p)}{q^{r}p^{t}}$ to conclude that
$\weakfixedspecific{\Gamma_{r}(q)\times\Delta_{t}(p)}{q^{r}p^{t}}$ would be mod $q$ acyclic also, which we know is not the case.
\end{proof}

The bulk of the rest of this section is to study the case $\weakfixedspecific{\Delta_{t}}{mp^t}$ where $m$ is not a power of a prime, with the aim of showing it is contractible. The strategy resembles that of Section~\ref{section: joins}, in that we replace the category $\weakfixedspecific{\Delta_{t}}{mp^t}$ with the category $\Uniform{\Delta_{t}}{mp^t}$ of
$\Delta_{t}$-fixed decompositions with uniform $\Delta_{t}$-isotropy, and then
we decompose $\Uniform{\Delta_{t}}{mp^t}$ into a union of two categories, the union of whose nerves gives the nerve of $\Uniform{\Delta_{t}}{mp^t}$.

\begin{definition}
\mbox{}\hfill
\begin{enumerate}
\item Let $\nontransitivespecific{\Delta_{t}}{mp^t}$ denote the subposet of $\Uniform{\Delta_{t}}{mp^t}$ consisting of decompositions $\lambda$ that have a nontransitive action of $\Delta_{t}$ on~$\class(\lambda)$.
\item Let $\properspecific{\Delta_{t}}{mp^t}$ denote the subposet of $\Uniform{\Delta_{t}}{mp^t}$ consisting of decompositions $\lambda$ such that the action of $\Delta_{t}$ on $\class(\lambda)$ is nontrivial; that is, $\lambda$ is not strongly fixed by~$\Delta_{t}$.
\end{enumerate}
\end{definition}

As in Section~\ref{section: joins}, we observe that if $\lambda$ is an object of $\nontransitivespecific{\Delta_{t}}{mp^t}$, then all refinements of $\lambda$ are objects in $\nontransitivespecific{\Delta_{t}}{mp^t}$ as well, and likewise $\properspecific{\Delta_{t}}{mp^t}$ contains all refinements of its objects. In addition, any $\lambda$ with uniform $\Delta_t$-isotropy is in one of these subspaces; hence the nerve of $\Uniform{\Delta_{t}}{mp^t}$ is the union of the nerves of the two subcategories, and we have a pushout diagram that is also a homotopy pushout:
\begin{equation}    \label{diagram: Delta_t pushout}
\begin{CD}
\nontransitivespecific{\Delta_{t}}{mp^t} \cap \properspecific{\Delta_{t}}{mp^t}
   @>>> \nontransitivespecific{\Delta_{t}}{mp^t}\\
@VVV  @VVV\\
\properspecific{\Delta_{t}}{mp^t}@>>> \Uniform{\Delta_{t}}{mp^t}
\end{CD}.
\end{equation}
\medskip

\begin{lemma}
The object spaces of the subcategories $\nontransitivespecific{\Delta_{t}}{mp^t}$ and $\properspecific{\Delta_{t}}{mp^t}$ are each unions of path components of the object space of $\Uniform{\Delta_{t}}{mp^t}$.
\end{lemma}
\begin{proof}
We apply Corollary~\ref{cor: J fixed path components}.
The action of the centralizer of $\Delta_{t}$ on the object space preserves the defining characteristics of $\nontransitive{\Delta_{t}}$ and $\properspecific{\Delta_{t}}{mp^t}$.
\end{proof}

The initial step to prove that the nerve of $\Uniform{\Delta_{t}}{mp^t}$ is contractible is to establish that the upper right-hand corner of
\eqref{diagram: Delta_t pushout} is contractible.

\begin{lemma} \label{lemma: nontransitive}
  Suppose that $\Delta_{t}$ is nontrivial and  acts on $\complexes^{m}\otimes\complexes^{p^t}$ as the tensor product of the trivial representation and the regular representation. Then $\nontransitivespecific{\Delta_{t}}{mp^t}$ has contractible
  nerve.
\end{lemma}

\begin{proof}
As usual, consider the inclusions
\begin{equation}  \label{eq: nontransitive subcat inclusions}
\isofixedspecific{\Delta_{t}}{mp^t}
     \hookrightarrow \strongfixedspecific{\Delta_{t}}{mp^t}
     \hookrightarrow \nontransitivespecific{\Delta_{t}}{mp^t}.
\end{equation}
The first map has the retraction functor
$\lambda\mapsto\isorefine{\lambda}{\Delta_{t}}$. The second map has the
retraction functor $\lambda\mapsto\glom{\lambda}{\Delta_{t}}$. (Note that
$\glom{\lambda}{\Delta_{t}}$ is necessarily proper by the definition
of~$\nontransitivespecific{\Delta_{t}}{mp^t}$.) Hence both inclusions of
\eqref{eq: nontransitive subcat inclusions} induce equivalences of nerves.
Finally, the action of $\Delta_{t}$ on $\complexes^{mp^t}$ is a multiple of the
regular representation, and is therefore not isotypic. Hence
$\isofixedspecific{\Delta_{t}}{mp^t}$ has a terminal object, namely the canonical
$\Delta_{t}$-isotypic decomposition of~$\complexes^{mp^t}$, and has contractible
nerve.
\end{proof}

Since we have shown that the upper right corner of diagram
\eqref{diagram: Delta_t pushout} has contractible nerve, to establish that $\Uniform{\Delta_{t}}{mp^t}$ has contractible nerve it is sufficient to show that the left-hand vertical map of~\eqref{diagram: Delta_t pushout} induces a homotopy equivalence on nerves. We will apply a topological version of Quillen's Theorem~A for categories internal to spaces, as stated and proved in \cite[Theorem 5.8]{Libman}. To check that the conditions of the cited theorem are satisfied, the first step is to look at the overcategories for objects in the lower left-hand corner of~\eqref{diagram: Delta_t pushout}.

\begin{proposition}   \label{prop: overcategory is contractible}
Let $\lambda$ be an object of $\properspecific{\Delta_{t}}{mp^t}$, and assume that $m>1$ is not a power of a prime. Let $\intersect$ denote the intersection
\[
\intersect
   =\nontransitivespecific{\Delta_{t}}{mp^t} \cap \properspecific{\Delta_{t}}{mp^t}.
\]
Then the overcategory $\intersect\downarrow\lambda$ has contractible nerve.
\end{proposition}
\begin{proof}
The category $\intersect\downarrow\lambda$ is the poset of refinements (not necessarily strict) of $\lambda$ that happen to be in $\nontransitivespecific{\Delta_{t}}{mp^t}$. In other words, $\intersect\downarrow\lambda$ contains objects $\mu\rightarrow\lambda$ such that the action of $\Delta_{t}$ on $\class(\mu)$ is not transitive.
If $\lambda$ is in $\nontransitivespecific{\Delta_{t}}{mp^t}$, then the identity morphism of $\lambda$ is a terminal object of $\intersect\downarrow\lambda$; therefore the category $\intersect\downarrow\lambda$ has contractible nerve.

So suppose that $\lambda$ is in $\properspecific{\Delta_{t}}{mp^t}$, but not in $\nontransitivespecific{\Delta_{t}}{mp^t}$. In particular, $\lambda$ has uniform
$\Delta_{t}$-isotropy~$\isogroupof{\lambda}$ properly contained in~$\Delta_{t}$, and $\class(\lambda)$ has a transitive action of~$\Delta_{t}$. To specify a refinement $\mu$ of $\lambda$
that lies in $\nontransitivespecific{\Delta_{t}}{mp^t}$, it is sufficient to choose one component $v\in\class(\lambda)$, and to specify an orthogonal decomposition of~$v$, call it $\mu_{v}$, such that
\begin{itemize}
\item $\mu_{v}$ is (weakly) stabilized by the action of $\isogroupof{\lambda}$ on~$v$,
\item $\isogroupof{\lambda}$ acts non-transitively on $\class(\mu_{v})$ (hence $\mu_{v}$ is proper), and
\item components of $\mu_{v}$ have uniform $\isogroupof{\lambda}$-isotropy.
\end{itemize}
The rest of $\mu$ is determined by transitivity of the $\Delta_{t}$-action on~$\class(\lambda)$. If we denote the dimension of $v$ by~$\dimv$, then the above shows that
\begin{equation}
\intersect\downarrow\lambda \cong \nontransitivespecific{\isogroupof{\lambda}}{r}.
\end{equation}

There are two cases: $\isogroupof{\lambda}=0$, and $\isogroupof{\lambda}$ is nontrivial. If $\isogroupof{\lambda}=0$, then $\nontransitivespecific{\isogroupof{\lambda}}{r}\cong\Lcal_{r}$. Because $r$ is a multiple
of~$m>1$, which is not the power of a prime, we know that $\Lcal_{r}$ is nonempty and has contractible nerve by Corollary~\ref{corollary: contractible}.
On the other hand, if $\isogroupof{\lambda}$ is not trivial, then
$\nontransitivespecific{\isogroupof{\lambda}}{r}$ has contractible nerve by Lemma~\ref{lemma: nontransitive}.
\end{proof}

The next condition to check in order to apply \cite[Theorem 5.8]{Libman} is the Reedy cofibrancy of an associated simplicial space.  Here, we need to consider an intermediate object between a topological category and its nerve, namely the \emph{simplicial nerve}, which is a simplicial diagram of spaces obtained by taking the levelwise nerve of the original topological category.  The nerve is then obtained by applying a realization functor, which gives a space.  The category of simplicial spaces can be given a Reedy model structure which we do not need here, but the properties of cofibrant objects in that model structure are sufficiently nice that they will be helpful in the results that follow.

\begin{lemma}   \label{lemma: Reedy cofibrant nerves}
The categories
$\intersect = \nontransitivespecific{\Delta_{t}}{mp^t}\cap\properspecific{\Delta_{t}}{mp^t}$
 and $\properspecific{\Delta_{t}}{mp^t}$ have simplicial nerves that
are Reedy cofibrant simplicial spaces.
\end{lemma}

\begin{proof}
We apply the criterion for Reedy cofibrancy in \cite[Proposition A.2.2]{Libman}. This criterion says that for a simplicial space $Y$ to be Reedy cofibrant, it is sufficient that $Y$ have an action of a compact Lie group $G$ such that in each simplicial dimension, $Y$ is a disjoint union of $G$-orbits, so that $Y/G$ is discrete. Here, we take $G=\CentIdent{\Delta_{t}}$, the identity component of the centralizer of $\Delta_{t}$ in $\Uof{mp^t}$. We observe the following.
\begin{itemize}
\item In order for a decomposition $\lambda$ to be in $\nontransitivespecific{\Delta_{t}}{mp^t}$, the action of $\Delta_{t}$ on $\class(\lambda)$ must be nontransitive, and $\lambda$ must have uniform $\Delta_{t}$-isotropy.
\item In order for a decomposition $\lambda$ to be in $\properspecific{\Delta_{t}}{mp^t}$, the action of $\Delta_{t}$ on $\class(\lambda)$ must be nontrivial, and $\lambda$ must have uniform $\Delta_{t}$-isotropy.
\end{itemize}
If $\lambda$ satisfies either (or both) of these conditions and $c\in\CentIdent{\Delta_{t}}$, then $c\lambda$ satisfies the same condition(s) as~$\lambda$.

But Corollary~\ref{cor: J fixed path components} tells us that the path components of both the object and morphism spaces of $\weakfixedspecific{\Delta_{t}}{mp^t}$ are orbits of~$\CentIdent{\Delta_{t}}$.
Applying the observations above tells us that both
$\nontransitivespecific{\Delta_{t}}{mp^t}\cap\properspecific{\Delta_{t}}{mp^t}$
and $\properspecific{\Delta_{t}}{mp^t}$ are unions of path components of
$\weakfixedspecific{\Delta_{t}}{mp^t}$, each of which is an orbit of~$\CentIdent{\Delta_{t}}$. We conclude that after taking the quotient by the
action of~$\CentIdent{\Delta_{t}}$, we have a discrete space in each simplicial dimension of the simplicial nerves, and
\cite[Proposition A.2.2]{Libman} now gives the desired result.
\end{proof}

\begin{proposition}   \label{proposition: apply Quillen A}
Suppose $ m>1$ is not a power of a prime. The inclusion of topological categories
\[
\intersect = \nontransitivespecific{\Delta_{t}}{mp^t} \cap \properspecific{\Delta_{t}}{mp^t}
\longrightarrow
\properspecific{\Delta_{t}}{mp^t}
\]
induces a homotopy equivalence of nerves.
\end{proposition}

\begin{proof}
We prove this result by applying Libman's version of Quillen's Theorem~A \cite[Theorem 5.8]{Libman} to the inclusion of opposite categories
\[
j \colon  \intersect^{op}
\longrightarrow
\left[\properspecific{\Delta_{t}}{mp^t}\right]^{op}.
\]
 Thus maps are now refinements of decompositions, i.e., if $\mu $ is a refinement of $ \lambda$, we have a map $\lambda \to \mu$ in the opposite category.

To apply the cited theorem, we need to know that the nerves of $\intersect^{op}$ and $\left[\properspecific{\Delta_{t}}{mp^t}\right]^{op}$ are Reedy cofibrant; since Reedy cofibrancy is preserved by taking opposites, we have this condition by Lemma \ref{lemma: Reedy cofibrant nerves}.

Next, we need to know that for any $\lambda \in\left[\properspecific{\Delta_{t}}{mp^t}\right]^{op}$, the nerve of the undercategory $\lambda \downarrow j$ is Reedy cofibrant; the proof can be handled just as that of Lemma~\ref{lemma: Reedy cofibrant nerves}. Specifically, the object space of $\lambda \downarrow j$ consists of $\mu \subseteq \lambda$, where $\mu $ is in~$\intersect$. Let $\UnIsotropyof{\lambda}$ denote the $\Un$-isotropy group of $\lambda$; then the group $\CentIdent{\Delta_{t}} \cap \UnIsotropyof{\lambda}$ acts transitively on each path component of $\lambda \downarrow j$, and taking the quotient by this action gives a discrete space.

Further, we need to know that for any $\lambda \in\left[\properspecific{\Delta_{t}}{mp^t}\right]^{op}$, the nerve of the undercategory $\lambda \downarrow j$ is contractible. But this undercategory is the same as the overcategory $\intersect\downarrow\lambda$, which was proved to be contractible in Proposition~\ref{prop: overcategory is contractible}.

Finally, we need to check that the inclusion $j$ is absolutely tame, a technical condition which we now recall from \cite[Definition 5.5]{Libman}.
Associated to the inclusion~$j$ is a bisimplicial space
$\mathcal{X}(j)$ whose space
$\mathcal{X}_{s,r}(j)$ of $(s,r)$-bisimplices consists of chains of refinements and coarsenings
\[
 \{ \lambda_s \rightarrow \dots \rightarrow \lambda_0 \leftarrow \mu_0 \leftarrow \dots \leftarrow \mu_r \},
 \]
where $\lambda_i \in \properspecific{\Delta_{t}}{mp^t}$ and $\mu_j \in \intersect$. There is a projection map
\[ \pi_{s,r} \colon \mathcal{X}_{s,r}(j) \to \Nerve_s(\properspecific{\Delta_{t}}{mp^t}) \]
which forgets the chain of $\mu$'s. We say that $j$ is \emph{absolutely tame} if $\pi_{s,r}$ is a Serre fibration for all $s,r\geq 0$ \cite[Definition 5.7]{Libman}. We will verify this condition.

A connected component of the object space of $\properspecific{\Delta_{t}}{mp^t}$, which is also the space of zero simplices $\Nerve_0(\properspecific{\Delta_{t}}{mp^t})$, is a $\CentIdent{\Delta_{t}}$-orbit; more precisely, the connected component of a decomposition $\lambda_0 $ in $ \Nerve_0(\properspecific{\Delta_{t}}{mp^t})$ is $ \CentIdent{\Delta_{t}}/  (\CentIdent{\Delta_{t}} \cap \UnIsotropyof{\lambda_0}) $, where $\UnIsotropyof{\lambda_0}$ is the $\Un$-isotropy group of $\lambda_0$. Consequently, we can determine that the connected component of $\lambda_\bullet = \{ \lambda_s \rightarrow \dots \rightarrow \lambda_0 \} \in \Nerve_s(\properspecific{\Delta_{t}}{mp^t})$ is $\CentIdent{\Delta_{t}}/  (\CentIdent{\Delta_{t}} \cap \UnIsotropyof{\lambda_\bullet})$, where
\[
\UnIsotropyof{\lambda_\bullet} = \UnIsotropyof{\lambda_0} \cap \dots \cap \UnIsotropyof{\lambda_s}.
\]
Similarly, the connected component of
\[ (\lambda_\bullet,\mu_\bullet) = \{ \lambda_s \rightarrow \dots \rightarrow \lambda_0 \leftarrow \mu_0 \leftarrow \dots \leftarrow \mu_t \} \in \mathcal{X}_{s,t}(j) \]
is
\[  \CentIdent{\Delta_{t}}/  \left(\CentIdent{\Delta_{t}} \cap \UnIsotropyof{\lambda_\bullet} \cap  \UnIsotropyof{\mu_\bullet}\right). \]
Restricted to a connected component of $\mathcal{X}_{s,t}(j)$, the map $\pi_{s,t}$ is the quotient induced by the inclusion of subgroups
\[  \CentIdent{\Delta_{t}} \cap \UnIsotropyof{\lambda_\bullet} \cap \UnIsotropyof{\mu_\bullet} \subseteq  \CentIdent{V} \cap \UnIsotropyof{\lambda_\bullet}, \]
and that quotient is a Serre fibration.

We have established that all conditions of \cite[Theorem 5.8]{Libman} are satisfied for the map $j$, so we conclude that it induces a homotopy equivalence on nerves. Thus the inclusion (before taking opposites of the categories)
\[ \intersect \to \properspecific{\Delta_{t}}{mp^t}\]
also induces a homotopy equivalence on nerves, as claimed.
\end{proof}

The work above showing that we can use Quillen's Theorem~A allows us to establish the following key result.

\begin{theorem}\label{theorem: fixed points Delta_t}
   Let $m>1$, and let $\Delta_{t}$ act on $\complexes^{mp^{t}}$ by $m$ copies of the regular representation. If $m$ is not a power of a prime, then $\weakfixedspecific{\Delta_{t}}{mp^t}$ is contractible.
\end{theorem}

\begin{proof}
If $t=0$ then the result follows from Corollary~\ref{corollary: contractible},
so we can assume that $t>0$ and $\Delta_{t}$ is nontrivial.

As in Section~\ref{section: joins}, we proceed by decomposing the category of interest.
Let $\Uniform{\Delta_{t}}{mp^t}$ denote the subposet of $\weakfixedspecific{\Delta_{t}}{mp^t}$ consisting of
objects with uniform $\Delta_{t}$-isotropy.
Exactly as in Proposition~\ref{proposition: inclusion of uniform subposet},
the inclusion $\Uniform{\Delta_{t}}{mp^t}\hookrightarrow\weakfixedspecific{\Delta_{t}}{mp^t}$
induces a homotopy equivalence on nerves. We will show that $\Uniform{\Delta_{t}}{mp^t}$ has a contractible nerve.

The object space of $\Uniform{\Delta_{t}}{mp^t}$ is a union of path components of $\Obj\weakfixedspecific{\Delta_{t}}{mp^t}$, a fact which follows from
Corollary~\ref{cor: J fixed path components}, since the action of the centralizer of $\Delta_{t}$ preserves the property of having uniform $\Delta_{t}$-isotropy.
 Therefore, the pushout diagram
\begin{equation}    \label{eq: diagram for contractibility}
\xymatrix{
\nontransitive{\Delta_{t}}\cap\properspecific{\Delta_{t}}{mp^t} \ar[r]\ar[d] & \nontransitive{\Delta_{t}} \ar[d]\\
\properspecific{\Delta_{t}}{mp^t} \ar[r] & \Uniform{\Delta_{t}}{mp^t}
}
\end{equation}
gives rise to a homotopy pushout diagram after applying the nerve functor.

The upper right-hand corner has contractible nerve,
by Lemma~\ref{lemma: nontransitive}. The left-hand vertical map induces a homotopy equivalence on nerves, by Proposition~\ref{proposition: apply Quillen A}. Thus $\Uniform{\Delta_{t}}{mp^t}$ must also be contractible, as needed.
\end{proof}

This result brings us at last to the assembly of the proof of the classification theorem. Recall that a coistropic subgroup of $\Gamma_{k}$ is one that has the form $\Gamma_{s}\times\Delta_{t}$ where $s+t=k$.

\begin{classificationtheorem}
\classificationtheoremtext
\end{classificationtheorem}

\begin{proof}
Suppose that $H$ is problematic. By Theorem~\ref{theorem: Non_contractible_implies_subgroup_Gamma_diag}, we may assume that $H$ is subconjugate to a subgroup of $\Gamma_k$ acting on $\complexes^{m}\otimes\Cpk$ by the standard action on $\Cpk$ and the trivial action on~$\complexes^{m}$. Hence $H\cong\Gamma_{s}\times\Delta_{t}$ for $s+t\leq k$.

To prove the converse for $m=1$, we must show that all subgroups of $\Gamma_{k}$ are in fact problematic.
If $H=\Gamma_{k}$, then $\weakfixedspecific{\Gamma_{k}}{p^{k}}$ is a wedge of spheres
(Proposition~\ref{proposition: fixed points of Gamma_k}), and hence has nontrivial mod~$p$ homology.
Since we have assumed that $S^{1}\subseteq H\subseteq\Gamma_{k}$, the quotient $\Gamma_{k}/H$ is a finite \pdash group, so we can apply Smith theory
to
 \[
 \left(\weakfixedspecific{H}{p^k}\right)^{\Gamma_{k}/H}
    =\weakfixedspecific{\Gamma_{k}}{p^k}.
 \]
We conclude that $\weakfixedspecific{H}{p^{k}}$ likewise has nontrivial mod~$p$ homology, and therefore is not contractible. Therefore $H$ is problematic.

Now suppose $m>1$ and $H\cong\Gamma_{s}\times\Delta_{t}\subseteq\Gamma_{k}$ is problematic. Let $r=n/p^{s+t}$.
We apply Proposition~\ref{proposition: reduction to Delta_t} to conclude that
because $\weakfixed{H}$ is not contractible, then the
$s$-fold suspension of $\weakfixedspecific{\Delta_{t}}{rp^t}$ is not contractible.
Hence $\weakfixedspecific{\Delta_{t}}{rp^t}$ is likewise not contractible.
The contrapositive of Theorem~\ref{theorem: fixed points Delta_t} implies that
$r$ is a power of a prime, and since $m\mid r$ and $m$ is coprime to~$p$, this means that $m=r=q^i$
for $i>1$ and $q$ a prime different from $p$. In particular, since $n=rp^{s+t}=mp^{k}$, we have $s+t=k$, so $H$ is coisotropic.

In summary, if $\weakfixed{H}$ is noncontractible, then
\begin{itemize}
\item $\weakfixedspecific{\Delta_{t}}{rp^t}$ is not contractible, and
\item $r=q^{i}$ for $i>0$ and $q$ a prime different from~$p$.
\end{itemize}
To finish, we need to know that
$\weakfixedspecific{\Delta_{t}}{q^{i}p^{t}}$ is not contractible, a result provided
by Proposition~\ref{prop: full Delta is problematic}.
\end{proof}

\newpage

\newcommand{\etalchar}[1]{$^{#1}$}
\providecommand{\bysame}{\leavevmode\hbox to3em{\hrulefill}\thinspace}
\providecommand{\MR}{\relax\ifhmode\unskip\space\fi MR }
\providecommand{\MRhref}[2]{%
  \href{http://www.ams.org/mathscinet-getitem?mr=#1}{#2}
}
\providecommand{\href}[2]{#2}

\end{document}